\documentclass{article}
\usepackage[english]{babel}
\usepackage[T1]{fontenc}
\usepackage[utf8]{inputenc}
\usepackage{amsthm}
\usepackage{amsmath}
\usepackage{amssymb}
\usepackage{palatino}
\usepackage{graphicx}
\usepackage[colorlinks=true, citecolor =blue]{hyperref}
\usepackage{dsfont}
\usepackage[all]{xy}
\usepackage[top=4cm,bottom=4cm,left=3cm,right=3cm]{geometry}
\usepackage{enumitem}
\usepackage{mathtools}
\usepackage{subcaption}
\usepackage{authblk}
\usepackage{bbm}
\usepackage{verbatim}
\newcommand{\R}{\mathds{R}}
\newcommand{\N}{\mathds{N}}
\newcommand{\Z}{\mathds{Z}}

\newcommand{\sg}{\sigma}
\newcommand{\p}{\partial}
\newcommand{\eps}{\varepsilon}
\newcommand{\Rp}{\mathcal{C}}

\newtheorem{theorem}{Theorem}[section]
\newtheorem{lemma}[theorem]{Lemma}
\newtheorem{definition}[theorem]{Definition}
\newtheorem{remark}[theorem]{Remark}
\newtheorem{proposition}[theorem]{Proposition}
\newtheorem{corollary}[theorem]{Corollary}
\usepackage{tikz}

\numberwithin{equation}{section}%

\title{An infinite-times renewal equation}
\author[a]{Xu'an Dou\thanks{Email :
dxa@pku.edu.cn}}
\author[b]{Benoît Perthame\thanks{Email : benoit.perthame@sorbonne-universite.fr}}
\author[c]{Chenjiayue Qi\thanks{Email :
jiayue@pku.edu.cn}}
\author[d]{Delphine Salort\thanks{Email : delphine.salort@sorbonne-universite.fr}}
\author[a]{Zhennan Zhou\thanks{Email : zhennan@bicmr.pku.edu.cn}}

\affil[a]{Beijing International Center for Mathematical Research, Peking University, Beijing, 100871, China.}
\affil[b]{Sorbonne Université, CNRS, Inria, Université Paris Cit\'e, Laboratoire Jacques-Louis Lions, 4 Place Jussieu, 75005 Paris, France.}
\affil[c]{Department of Philosophy, Peking University, Beijing, 100871, China.}
\affil[d]{Sorbonne Université, CNRS, Laboratoire de Biologie Computationnelle et Quantitative, UMR 7238. 4 Place Jussieu, 75005 Paris, France.}

\date{\today}

\begin{document}

\maketitle

\begin{abstract}

In neuroscience, the time elapsed since the last discharge has been used to predict the probability of the next discharge. Such predictions can be improved taking into account the last two discharge times, and possibly more. Such multi-time processes arise in many other areas and there is no universal limitation on the number of times to be used. This observation leads us to study the infinite-times renewal equation as a simple model to understand the meaning and properties of such partial differential equations depending on an infinite number of variables.

We define two notions of solutions, prove existence and uniqueness of solutions, possibly measures. We also prove the long time convergence, with exponential rate, to the steady state in different, strong or weak, topologies depending on assumptions on the coefficients.
\end{abstract}

\vskip .7cm

\noindent{\makebox[1in]\hrulefill}\newline

\section{Introduction}

A number of physical and biological processes take into account several successive events, such as the spike times of a neuron, seismic sequences, and more generally multi-time renewal processes. Using the language of partial differential equations (PDE)  to describe the probability distribution  of such processes at time $t$, which we denote by $n_\infty(t, s_1,s_2,...)$, we are interested in the infinite-times renewal equation, where the renewal rate is denoted by $p_\infty$. It is set in the domain $0\leq s_1\leq s_2\leq ...$, where $s_i$ refers to the elapsed time of the $i$-th recent event counting from present to past, and can be written as,
\begin{equation}
\label{eq:infinite}
\left\{
\begin{matrix*}[l]
\partial_t n_\infty+ \displaystyle \sum_{i=1}^\infty \partial_{s_i} n_\infty +p_\infty(s_1,s_2,...) n_\infty=0, 
\\[5pt]
n_\infty(t,s_1=0,s_2,...)= \int_{u=0}^\infty p_\infty(s_2,..., s_{K}, ...,u)n_\infty(t,s_2,..., s_{K},...,u)\,du.
\end{matrix*}
\right.
\end{equation}
This writing is formal since there is no 'last variable' $u$ to integrate with. Our aim is to define precisely the notion of solutions, in particular the meaning of the boundary condition, and to prove uniqueness. We build these solutions  taking  the limit as $N\to \infty$ of the $N$-times renewal equation,
\begin{equation}
\label{eqKTR}
\left\{
\begin{matrix*}[l]
\partial_t n_N+ \displaystyle \sum_{i=1}^N \partial_{s_i} n_N +p_N([s]_N)n_N=0,
\\[5pt]
n_N(t,s_1=0, s_2,..., s_N)= \int_{u=0}^\infty p_N( s_2,..., s_N,u) \, n_N(t, s_2,..., s_N,u)\,du.
\end{matrix*}
\right.
\end{equation}
For the applications we have in mind, the existence of a limit justifies to restrict the physical description to an arbitrary value $N$, since the precise value does not change too much after some range. Furthermore, we show that this approximation $N\to \infty$ holds uniformly in time.

\textbf{Motivations from biology} The boundary condition in \eqref{eqKTR} means that after renewal, a new time is initiated and all labels are shifted of $+1$ thus the 'oldest' is forgotten and the penultimate becomes the last. This process can be used in neuroscience where $n(t,s)$ describes the probability to find a neuron with time $s$ elapsed since the last discharge~\cite{PGSchwalger, PPD, PPD2,canizo2019asymptotic}. When the firing rate depends on the $N$ last events, we find the $N$-times renewal equation as proposed for $N=2$ in~\cite{perthameST2022}. The $2$-times renewal equation has also been used for establishing  the efficacy of contact tracing during an epidemic spread~\cite{ferretti2020quantifying}. In these applications, the variable $s_i$ refers to age but in many other areas multiple structures occur with different velocities, leading to more general equations under the form
\[
\partial_t n +  \sum_{i=1}^N \partial_{s_i}\big[ g_i([s]_N) n\big] +p_N([s]_N, n_N)n_N=0.
\]
See~\cite{Jaeger2015}  for instance. Our analysis could cover such general settings but for the sake of simplicity, we keep the forms \eqref{eq:infinite}--\eqref{eqKTR}.

In neuroscience, this aspect is tightly related to more elaborate descriptions of the stochastic processes underlying neural activity, in particular the Hawkes and Wold processes.

\textbf{Relation to the Hawkes and Wold processes}
Our generalization of renewal equations to infinite-times shares similarities with the language of Hawkes process, as they appear in computational neuroscience~\cite{Pernice_2011}, genome analysis~\cite{Patricia_2010} and financial analysis~\cite{Luc_2009} etc. It is likely that infinite-times renewal equations can be also used in these questions. A Hawkes process can be understood as a point process ${\cal N}= \{T_i\}_{i\in\Z}$ where $T_i\in\R$ is the time when the $i$-th renewal happens. For this process, and in the simplest form, the renewal rate at time $t$ is defined as $\sum_{i\in\Z}h(t-T_i)\mathbbm{1}_{T_i<t}$,  which implies that the renewal rate may depend on the infinite renewal times in the past. Among the numerous  results dealing with Hawkes processes, let us mention their distribution~\cite{chevallier2015microscopic} and their long time convergence~\cite{Costa_Hawkes_2020,GRAHAM_Hawkes_2019}. 

Apart from Hawkes processes, Wold processes are also connected to~\eqref{eq:infinite}, where a particle renewal rate depends on a fixed number of renewals in the past (see~\cite{chevallier2015microscopic}). They could possibly correspond to the finite-times renewal equation. However, the possibility of a Wold process depending on infinite renewal times in the past is still largely open. 

\textbf{Content of the paper.} In order to study the problem \eqref{eq:infinite}, we need to  recall, following \cite{perthameST2022}, properties of the $N$-times renewal equation, such as the $L^1$ non-expansion property and  $L^\infty$ bounds. Especially we prove a tightness estimate (Lemma~\ref{lemma:tightness}) uniformly in dimension~$N$, which facilitates later weak convergence results (Theorem~\ref{th:weak_limit_of_finite}).

Based on these results, our first goal is to prove that there is a limit as $N \to \infty$ of the $N$-times renewal equation. On the one hand, in the spirit of the BBGKY hierarchy in particles systems, see \cite{CIP1994}, this limit can be understood as a hierarchy of marginals $n_\infty^{(K)}$ each of them satisfying an equation depending on the next one, see Theorem~\ref{th:uniqueness}. We call them 'hierarchy solution', see Definition~\ref{def:hierarchy_weak_solution}. On the other hand, using the Kolmogorov extension theorem, it can also be understood as a measure $n_\infty$ in infinite dimensions (Theorem~\ref{th:weak_limit_of_finite}), which satisfies in a weak sense the infinite-dimensional equation (Definition \ref{def:MSI}). These two points of view turn out to be equivalent and our main result is well-posedness of solutions  of Equation \eqref{eq:infinite}.

Our second goal is to study the long time behaviour of these solutions to the infinite dimensional problem. For that, we use two different approaches. For the $N$-times renewal equation, we can use the Doeblin method to prove their exponential $L^1$-convergence, see Theorem~\ref{thm:ststN}. But as the dimension goes to infinity, the $L^1$-convergence rate will go to zero. However, if the renewal rate $p_N$ converges fast enough to $p_\infty$,  one can prove some uniformity-in-time convergence as $N\to \infty$ and also 
 some uniformity in $N$ of the long term behaviour, still in $L^1$, see Theorem~\ref{thm:uni-n-time-concrete}. The $L^1$ topology is hardly compatible with working with the infinite dimensional measure $n_\infty$. Therefore we also  use the Monge-Kantorovich distance to work in the weak topology of measures. We establish a long term convergence result for finite systems (Theorem \ref{th:finite_monge_convergence}) as well as infinite systems (Theorem~\ref{th:infinite_MK_convergence}). Again, these results are uniform in $N$ and rely on some, uniform in $N$,  smallness condition on a Lipschitz norm of $p_N$.

\textbf{Outline of the paper.} In Section \ref{Ntimes}, we prove several properties for the finite-times renewal equations. Departing from these results for finite-times renewal equations, we use two approaches to handle the weak solution of the infinite-times renewal equation, where the first approach is in Section \ref{Sec:hierarchy}-\ref{sec:uniformError} involving strong topology is the 'hierarchy solution', and the second approach, in Section~\ref{sec:MeasureSol}-\ref{sec:MK-BIG}, involves weak topology and 'measure solutions' thanks to the Kolmogorov extension theorem and Monge-Kantorovich distance.

\section{Notations} 
\label{sec:notations}
\paragraph{General notations.}
As mentioned earlier, we consider variables which satisfy $s_i\leq s_{i+1}$, and thus we introduce the domains
\begin{equation}\label{def-CN}
\mathcal{C}_N= \{0\leq s_1\leq...\leq s_N\}, \qquad \quad   \mathcal{C}_\infty=\{0\leq s_1...\leq s_N\leq...\}.
\end{equation}

We use the notation $[s]_N$ to denote the vector $(s_1,...,s_N)\in\Rp_N$, $s_i$ is the $i$-th entry in $[s]_N$, $d[s]_N$ to denote the Lebesgue measure on $\Rp_N$. More generally, we set
\[
[s]_{K,N}=(s_K,...,s_N)\in\Rp_{N-K+1} \qquad \text{for} \qquad K\leq N.
\] 
In the same way, $[s]_{\infty}$ corresponds to  an infinite vector. We use $\tau$ as the shift operator, where
\begin{equation}\label{def:shift_operator}
\tau(s_1,s_2,...,s_N)=(0,s_1,s_2,...,s_{N-1}),\qquad \tau(s_1,s_2,...,s_N,...)=(0,s_1,s_2,...,s_N,...)
\end{equation}
For $K,N=1,2,...,\infty$ with $K\leq N$, we use $n_N$ to denote the solution of the $N$-dimensional problem, while $n_N^{(K)}$ denotes the first $K$ marginal of $n_N$
\begin{equation}
\label{Marginales}
n_N^{(K)}  (t,[s]_{K})= \int_{s_{K}\leq s_{K+1}\leq...\leq s_N} n_N(t,s_1,..., s_{N})\,d[s]_{K+1, N}.
\end{equation}
It is obvious that the marginals satisfy, for $K<N$,
\begin{equation}
\label{MarginalesNK}
n_N^{(K)}  (t,[s]_{K})= \int_{s_K}^\infty n_N^{(K+1)} (t,s_1,..., s_{K+1})\,ds_{K+1}.
\end{equation} 
We use $\mu_N$ to denote a measure on $\mathcal{C}_N$, and in particular $\mu_\infty$ is a measure on $\mathcal{C}_\infty$.  We also use $\mu^{(K)}$ to denote the $K$ marginal of a distribution $\mu$.
\begin{definition} [Consistent sequences] \label{def:consistent}
A sequence of measure $(\mu_N)_{N\geq 1}$ is said to be consistent if its marginals satisfy $\mu_K=\mu_N^{(K)}$ for all $K < N$. In case of a measure $\mu_{\infty}$, it means  
$\mu_\infty^{(K)}=\int_0^\infty \mu_\infty^{(K+1)} ds_{K+1}$ for all $K\geq 1$.
\end{definition}

We frequently use several functions, among them let us mention, for $N=1,2,...,\infty$ the moment function (for tightness) 
\begin{equation} \label{def:weight}
\sg_N([s]_N) = \sum_{i=1}^N \frac{s_i}{2^{i}}.
\end{equation}
In a given a measurable space $\mathcal{X}$,  we use $\mathcal{M}(\mathcal{X})$ to denote the space of signed measures with finite total measure, and we denote the total variation norm (TV-norm) on $\mathcal{M}(\mathcal{X})$ as $\| \cdot \|_{{\mathcal M}^1}$ instead of $\lVert \cdot\rVert_{TV}$. The set of probability measures is denoted by $\mathcal{P}(\mathcal{X})$, which is a subset of $\mathcal{M}(\mathcal{X})$.
We use $C_w([a,b];\mathcal{M}(\mathcal{X}))$ to define the space of continuous functions with values in $\mathcal{M}(\mathcal{X})$ endowed with its weak topology and we use the naming {\em weakly continuous}. We use $C_b^i(\mathcal{X})$ to define the space of bounded $i$th-differentiable functions.

\paragraph{Assumptions and weak solutions.}
We assume the particular form of the renewal rate,
\begin{equation} \label{as:RR1}
p_N= \sum_{i=1}^N \varphi_i(s_1,..., s_i), \qquad  \varphi_i >0, \qquad   \sum_{i=1}^\infty \| \varphi_i \|_{\infty}\leq a_+ < \infty,
\end{equation}
where $\lVert \cdot\rVert_{\infty}$ is the $L^\infty$ norm. We also define
\begin{equation}\label{def-eps}
    \eps_{K,N}:=\sum_{i=K+1}^{N}\|\varphi_i\|_{\infty},\qquad \quad \eps_{K}:=\eps_{K,\infty}.
\end{equation}
We also need a lower control and we assume
\begin{equation} \label{as:RR2}
0< a_- \leq p_N([s]_N)\leq a_{+}.
\end{equation}

\begin{definition}[Weak solution of the N-times renewal equation]
\label{def:MSI_finite}
We say that  $n_{N}\in C\big([0,+\infty);L^1(\Rp_{N})\big)$ is a weak solution of the N-times renewal equation~\eqref{eqKTR} if for all $T>0$ and all test functions $\psi\in C_b^1([0,T]\times \Rp_N)$, with $\tau[s]$ as the shift operator  defined in \eqref{def:shift_operator}, 
\begin{equation}\label{eq:measure-solution-finite}
\begin{split}
\int_0^{T} \! \! \int_{\mathcal{C}_N}  n_{N}(t,[s]_{N})&\Big[- \p_t\psi(t,[s]_{N})- \sum_{i=1}^{N}\p_{s_i}\psi(t,[s]_{N})+\big(\psi(t,[s]_{N})-\psi(t,\tau [s]_{N})\big)p_{N}([s]_{N})\Big] d[s]_{N} dt
\\
=& -\int_{\Rp_N} \psi(T,[s]_{N})n_{N}(T,[s]_{N})d[s]_{N}+\int_{\Rp_N} \psi(0,[s]_{N})n_{N}(0,[s]_{N}) d[s]_{N} .
\end{split}
\end{equation} 
\end{definition}

For the infinite case, we give definitions in the corresponding sections.
\\

We use $\mathcal{N}_N([s]_{N-1})$ to denote the boundary data of the $N$-times problem
\begin{equation}\label{eq:boundary_condition_1}
\mathcal{N}_N(t,[s]_{N-1})=\int_{s_{N-1}}^{+\infty} n_{N}(t,[s]_N)p_{N}([s]_N)ds_N, \qquad \mathcal{N}_{\infty}([s]_{\infty})=p_{\infty}([s]_{\infty})n_{\infty}([s]_{\infty}) .
\end{equation} We note that $\mathcal{N}_{\infty}$ does not involve an integration since there is no history variable to forget after a renewal. We use $n_{N}^{*}$ to denote the stationary solution of the $N$-dimensional problem, while $n_{N}^{*,(K)}$ refers to its  first $K$ marginal. 

\section{The  $N$-times renewal equation}   \label{Ntimes}

Since we see the infinite-times renewal equation as the limit of the $N$-times equation, our first purpose is to prove uniform estimates in $N$ for the solutions of Equation~\eqref{eqKTR}. We complete Equation \eqref{eqKTR} with an initial data which is a probability density
\begin{equation} \label{as:ID}
n_N(0,[s]_N)\geq 0, \quad \int_{\Rp_N} n_N(0,[s]_N) \,d[s]_{N}=1, \quad n_N(0,[s]_N) \text{ is supported in } \Rp_N.
\end{equation}

The approaches to the well-posedness of the $2$-times renewal equation apply here, see~\cite{perthameST2022}. For any $1\leq N<+\infty$, any $p_N$ satisfying Assumption \eqref{as:RR2} and any initial data $n_N(0)$ supported in $\mathcal{C}_N$, Equation \eqref{eqKTR} has a unique weak solution $n_N(t)$ in the sense of Definition \ref{def:MSI_finite}. In particular, solutions belong to $C\big([0,+\infty); L^1(\Rp_N)\big)$ and the properties \eqref{as:ID} are propagated, for all times $t>0$, we still  have 
\begin{equation}\label{eq:conservation}
n_N(t,[s]_N) \geq 0, \qquad   \int_{\Rp_N} n_N(t,[s]_N) \,d[s]_{N} =1,  \quad n_N(t,[s]_N) \text{ is supported in } \Rp_N.
\end{equation}

\subsection{Marginales in Equation~\eqref{eqKTR}}

We are going to prove uniform in $N$ estimates for the marginales of $n_N$  based on the equation they satisfy, namely
\begin{equation}
\label{eqNK}
\left\{
\begin{matrix*}[l]
\partial_t n_N^{(K)}+ \displaystyle \sum_{i=1}^K \partial_{s_i} n_N^{(K)} +p_K([s]_K )n_N^{(K)} + E_N^{(K)}(t,[s]_K) =0, 
\\[5pt]
n_N^{(K)}(t,s_1=0,[s]_{2,K})= \int_{0}^\infty \big[ p_K n_N^{(K)}+ E_N^{(K)} \big] (t,[s]_{2,K}, u)  du,
\end{matrix*}
\right.
\end{equation}
with the coupling term
\[
E_N^{(K)}(t,[s]_K) =  \sum_{i=K+1}^N  \int_{s_{K+1}=0}^\infty \! ... \! \int_{s_{N}=0}^\infty  \varphi_i([s]_i)  n_N (t, [s]_N) \,d[s]_{K+1, N} .
\]
We prefer to write the coupling terms above in terms of the marginales
\begin{equation}
\label{couplingNK}
E_N^{(K)}(t,[s]_K) = \sum_{i=K+1}^N \int_{s_{K+1}=0}^\infty \! ... \! \int_{s_{i}=0}^\infty  \varphi_i([s]_i)  n_N^{(i)}  (t, [s]_i) \,d[s]_{K+1, i}   , 
\end{equation}
and we have a pointwise bound
\begin{equation}
    0\leq E_N^{(K)}(t,[s]_K) \leq \left( \sum_{i=K+1}^N  \| \varphi_i  \|_\infty \right) n_N^{(K)}(t,[s]_K),
\end{equation}
and we conclude the $L^1$ estimate
\begin{equation} \label{Error_estimate}
\| E_N^{(K)}(t,\cdot) \|_{L^1(\Rp_K)} \leq \eps_{K,N}:= \sum_{i=K+1}^N  \| \varphi_i \|_\infty .
\end{equation}

\subsection{Uniform $L^\infty$ estimates}
\label{sec:linfnty}

In order to prove the $L^\infty$ estimates, which are uniform in $N$, we introduce a subsolution. With  definition \eqref{def-eps} for $\eps_{K}$, we consider the stationary solution $\overline n_K ([s_K]) >0$ of the $K$-times renewal equation
\[ \begin{cases}
 \displaystyle \sum_{i=1}^K \partial_{s_i} \overline n_K +\big(\eps_{K} + p_K([s]_K) \big) \overline  n_K=0,
\\[5pt]
\overline n_K(s_1=0,[s]_{2,K})= \int_{u=0}^\infty \big[ p_K([s]_{2,K}, u )+  \displaystyle  \eps_{K} \big]\;  \overline  n_{K}([s]_{2,K}, u)  du.
\end{cases} 
\]

\begin{lemma}[Locally uniform $L^\infty$ bounds]  \label{pr:iniformbdd} 
We assume  \eqref{as:RR1}, \eqref{as:ID} and  that for some $K \leq N$
\begin{equation} \label{as:IDbdd}
n^{(K)}_N(0,[s]_K)\leq C_K \overline  n_K([s]_K).
\end{equation}
Then, for all $t>0$,  we also have 
\begin{equation} \label{as:bdd}
 n^{(K)}_N(t,[s]_K) \leq C_K   e^{\eps_{K} t} \overline n_K ([s]_K) .
\end{equation}
\end{lemma}

\begin{proof}
Departing from \eqref{MarginalesNK}, the $K$  marginal also satisfies
\[ \begin{cases}
\partial_t n_N^{(K)}+ \displaystyle \sum_{i=1}^K \partial_{s_i} n_N^{(K)} +p_K([s]_K )n_N^{(K)} \leq 0, 
\\[5pt]
n_N^{(K)}(t,s_1=0,[s]_{2,K}) \leq \int_{0}^\infty \big[ p_K([s]_{2,K}, u )+  \displaystyle  \sum_{i=K+1}^N \| \varphi_i \|_\infty \big] n_N^{(K)}(t,[s]_{2,K}, u)  du .
\end{cases} 
\]
This shows that $n_N^{(K)}$ is a subsolution of the following equation\begin{equation}\label{eq:equation_with_extra_source}
    \begin{cases}
 \displaystyle \p_tn_K+\sum_{i=1}^K \partial_{s_i}  n_K + p_K([s]_K)   n_K=0,
\\[5pt]
 n_K(s_1=0,[s]_{2,K})= \int_{u=0}^\infty \big[ p_K([s]_{2,K}, u )+  \displaystyle  \eps_{K} \big]\;    n_{K}([s]_{2,K}, u)  du.
\end{cases} 
\end{equation}
and thus is less than $C_K e^{\eps_K t} \overline n_K$ which is a solution.
\end{proof}

\subsection{Uniform tightness estimate}

 A first step to prove that this family of probability measures $\{n_{N}(t)\}_{N}$ weakly converges, we establish tightness. For that purpose, the use of the weight function $\sigma_N$ defined by \eqref{def:weight} is crucial in achieving a balance between dimension and  control of the decay of solutions at infinity. It helps to prioritize the first several variables while it also takes account properly  all variables so  that the derived estimates are uniform in $N$.

\begin{lemma}\label{lemma:tightness} With Assumptions \eqref{as:RR2} and \eqref{as:ID}, for all $t\geq 0$, we have,
\[
\int_{\Rp_N}  \sg_N([s]_N) n_N(t,[s]_N)  d[s]_N \leq  e^{-\frac{a_- t}{2}}  \int_{\Rp_N} \sg_N([s]_N) n_N(0,[s]_N)  d[s]_N +(1- e^{-\frac{a_- t}{2}})\frac{2}{a_-} .
\]
\end{lemma}

\begin{proof}
Based on Definition \ref{def:MSI_finite} of weak solutions and using the density argument, we can use $\sigma_N$ as a test function and write, 
\begin{align*}
\frac{d}{dt}& \int_{\Rp_N}  \sg_N([s]_N) n_N(t,[s]_N) d[s]_N +  \int_{\Rp_N}  \sg_N ([s]_N) p_N([s]_N)n_N(t,[s]_N) d[s]_N 
\\
=& \sum_{i=1}^{N} \int_{\Rp_N} n_N(t,[s]_N)  \frac{\partial \sg_N}{\partial s_i}([s]_N) d[s]_N  + \int_{\Rp_{N-1}} \sg_N(0, [s]_{2,N}) n_N (t,0, [s]_{2,N}) d[s]_{2,N}.
\end{align*}
Since the boundary condition gives
\begin{equation*}
n_N(t,0,[s]_{2,N})=\int_0^{\infty}p_N([s]_{2,N},u)n_N(t,[s]_{2,N},u)du,
\end{equation*}
we have
\begin{align*}
\frac{d}{dt} \int_{\Rp_N}&  \sg_N([s]_N) n_N(t,[s]_N) d[s]_N +  \int_{\Rp_N}  \sg_N ([s]_N) p_N([s]_N)n_N(t,[s]_N) d[s]_N 
\\
=& \sum_{i=1}^{N}2^{-i} \int_{\Rp_{N}} n_N(t,[s]_N) d[s]_N  + \int_{\Rp_{N-1}} \sg_N(0, [s]_{2,N}) \int_0^\infty p_N([s]_{2,N},u)n_N (t, [s]_{2,N}, u ) du d[s]_{2,N}.
\end{align*}
Using the  normalization condition for $n_N$, we further write
\begin{equation*}
\begin{split}
\frac{d}{dt} \int_{\Rp_N} & \sg_N([s]_N) n_N(t,[s]_N) d[s]_N +  \int_{\Rp_N}  \sg_N ([s]_N) p_N([s]_N) n_N(t,[s]_N) d[s]_N\\
\leq& 1+ \frac {1}{2}   \int_{\Rp_N} \sg_N([s]_{2,N},u)  p_N([s]_{2,N},u) n_N (t, [s]_{2,N}, u )  d[s]_{2,N} du.
\end{split}
\end{equation*}
Then we replace $u$ with $s_N$ to write the equation as
\begin{equation*}
\begin{split}
\frac{d}{dt} \int_{\Rp_N} & \sg_N([s]_N) n_N(t,[s]_N) d[s]_N +  \int_{\Rp_N}  \sg_N ([s]_N) p_N([s]_N)n_N(t,[s]_N) d[s]_N\\
\leq& 1+\frac{1}{2}\int_{\Rp_N}\sigma_N([s]_N)p_N([s]_N)n_N(t,[s]_N)d[s]_N.
\end{split}
\end{equation*}
Consequently 
\[
\frac{d}{dt} \int_{\Rp_N}  \sg_N([s]_N) n_N(t,[s]_N) d[s]_N +  \frac {1}{2}   \int_{\Rp_N}  \sg_N ([s]_N) p_N([s]_N) n_N(t,[s]_N) d[s]_N  \leq 1.
\]
With Assumption \eqref{as:RR2}, we have,
\begin{equation*}
\frac{d}{dt} \int_{\Rp_N}  \sg_N([s]_N) n_N(t,[s]_N) d[s]_N\leq -\frac{a_-}{2}\int_{\Rp_N} \sg_N([s]_N) n_N(t,[s]_N) d[s]_N+1.
\end{equation*}
Use the Gronwall lemma and we conclude.
\end{proof}
%
\subsection{Steady states and long term behavior}
\label{sec:StStN}

Following recent literature \cite{canizo2019asymptotic,perthameST2022}, we use Doeblin's method (see Appendix \ref{sec:Doeblin_Theorem}) to build a steady state for the $N$-times renewal equation. For that purpose, given $t^* > 0$ and with the notations of Assumption~\eqref{as:RR2},  we define the numbers
\begin{equation}\label{LTB:numbers}
\left\{
\begin{split}
& \alpha_N(t^*):=a_-^{N} \int_{0\leq s_N\leq t^*}\frac{1}{(N-1)!}s_N^{N-1} e^{-s_Na_+}ds_N,
\\[5pt]
& \lambda_N (t^*) :=-\frac{\ln(1-\alpha_N(t^*))}{t^*},\qquad \quad c_N(t^*):=\frac{1}{1-\alpha_N(t^*)}.
\end{split}
\right.
\end{equation}
By a standard calculation, we can see that $0<\alpha_N(t^*)<1$ for all $t^*>0$. Consequently we have $\lambda_N(t^*)>0$ and $c_N(t^*)>1$. Our main result is  
\begin{theorem} [Steady state, convergence and tightness]\label{thm:ststN}
 Fix the dimension $N$ and assume \eqref{as:RR2}, \eqref{as:ID}. Then, the solution $n_N(t,[s]_N)$ of Equation~\eqref{eqKTR} satisfies
\begin{equation} \label{eq:lowerbd}
n_N(t,[s]_N)\geq a_{-}^{N}e^{-s_{N}a_+},\qquad \textup{for} \quad s_1\leq ...\leq s_{N}\leq t .
\end{equation}
Also, $n_N(t,[s]_N)$ converges exponentially in the $L^1$-norm to the unique stationary state $n_N^*([s]_N)$ and, for all $t^*>0$ we have
\begin{equation}\label{eq:-converge-N-long-time}
\lVert n_N(t) -n_N^*\rVert_{L^1(\Rp_N)}\leq c_N(t^*)e^{-\lambda_N(t^*)\, t}\lVert n_N(0)-n_N^*\rVert_{L^1(\Rp_N)},\quad \forall \,t\geq 0 .
\end{equation}

Furthermore, the tightness estimate holds
\begin{equation}\label{uniform-moment-steady}
\int_{\Rp_N}\sigma_N([s]_N)n_N^{*}(d[s]_N)\leq \frac{2}{a_{-}}.
\end{equation}
\end{theorem}

At this stage, the steady state $n_N^*$ is a $L^1$ function and it satisfies the equation
\begin{equation} \label{eq:stst}
\left\{
\begin{split}
&\sum_{i=1}^{N}\p_{s_i}n_N^{*}+p_N n_N^{*}=0, \\
&n_N^{*}(0,[s]_{N-1})=\int p_N([s]_N) n_N^{*}([s]_{N-1},s_N) ds_N,  \quad  n_N^{*} \; \text{is supported in} \;  {\mathcal C}_N.
\end{split}
\right.
\end{equation}
It is possible to prove better regularity but, being interested in the limit $N \to \infty$, we do not go in this direction.

\begin{remark}
Although Theorem \ref{thm:ststN} gives the long time behavior of each $N$-times equation, the convergence rate $\lambda_N$ in~\eqref{eq:-converge-N-long-time} is \textit{not} uniform in $N$, which brings essential difficulties to study the long time behavior for the infinite-times problem. This issue is studied in Sections~\ref{sec:uniformError} and~\ref{sec:MK-BIG}.
\end{remark}

\begin{proof}
{\em First step. The estimate \eqref{eq:lowerbd}.}
We use  the definition \eqref{eq:boundary_condition_1} of the boundary renewal flux $\mathcal{N}_N(t,[s]_{N-1})$ defined by the boundary condition. By the method of characteristics for Equation~\eqref{eqKTR}, we can write the expression of $n_N(t,[s]_N)$ as
\begin{equation*}
n_N(t,[s]_N)=e^{-\int_0^{s_1}p_N(u,[s-s_1+u]_{2,N})du}\mathcal{N}_N(t-s_1,[s-s_1]_{2,N}) \qquad \text{for}\quad  t\geq s_1,   
\end{equation*}
and, changing variables in definition \eqref{eq:boundary_condition_1} of the boundary condition, we have
\begin{equation*}
\mathcal{N}_N(t-s_1,[s-s_1]_{2,N})=\int_{s_N\leq s_1'}n_N(t-s_1,[s-s_1]_{2,N},s_1'-s_1)p_N([s-s_1]_{2,N},s_1'-s_1)ds_1'.
\end{equation*}
which gives, on $\mathcal C_N$ and  for $t\geq s_1$, 
\begin{equation*}
\begin{split}
n_N(t,[s]_N)=&e^{-\int_0^{s_1}p_N(u,[s-s_1+u]_{2,N})du}\\
&\quad \int_{s_N\leq s_1'}n_N(t-s_1,[s-s_1]_{2,N},s_1'-s_1)p_N([s-s_1]_{2,N},s_1'-s_1)ds_1'
\end{split}
\end{equation*}
We can now use Assumption \eqref{as:RR2} to derive  from the above equation the lower bound
\begin{equation}\label{eq:lower_bound_4}
n_N(t,[s]_N)\geq a_-e^{-s_1a_+}\int_{s_N\leq s_1'}n_N(t-s_1,[s-s_1]_{2,N},s_1'-s_1)ds_1' , \qquad  t\geq s_1.
\end{equation}
This procedure is essential toward estimate \eqref{eq:lowerbd}, which states a lower bound for $n_N(t)$ independently of $n_N(t')$ for $t'<t$. In the following iterations, at each step we integrate one more variable until we find the total integral on $n_N$ which is known to be equal to $1$.
\\[3pt]
{\em Second step. Iterations.} For the second iteration, we use Equation \eqref{eq:lower_bound_4} to control from below $n_N(t-s_1,[s-s_1]_{2,N},s_1'-s_1)$. This means we replace $t$ by $t-s_1$, $[s]_N$ by $([s-s_1]_{2,N},s_1'-s_1)$ and, since the notation $s'_1$ is used, we call $s'_2$  the integration variable in the right hand side of \eqref{eq:lower_bound_4}. This gives, for~$t-s_1\geq s_2-s_1$
\begin{equation*}
n_N(t-s_1,[s-s_1]_{2,N},s_1'-s_1)\geq a_-e^{-(s_2-s_1)a_+}\int_{s_N\leq s_1'\leq s_2'}n_N(t-s_2,[s-s_2]_{3,N},[s'-s_2]_2)d[s']_2.
\end{equation*}
Combining this equation with \eqref{eq:lower_bound_4}, we find
\begin{equation}
n_N(t,[s]_N)\geq a_-^2e^{-s_2a_+}\int_{s_N\leq s_1'\leq s_2'}n_N(t-s_2,[s-s_2]_{3,N},[s'-s_2]_2)d[s']_2, \qquad \text{for} \quad t\geq s_2.
\end{equation}
As one can see, we have now integrated in the last two variable of $n_N$.  

Repeat this iteration and we will eventually have,
\begin{equation*}
n_N(t,[s]_N)\geq a_-^Ne^{-s_Na_+}\int_{s_N\leq s_1'\leq ...\leq s_N'}n_N(t-s_N,[s'-s_N]_N)d[s']_N \qquad \text{for} \quad t\geq s_N, 
\end{equation*}
where for the previous data $n_N(t-s_N)$ we integrate over the entire admissible set $\mathcal{C}_N$. By the probability normalization~\eqref{eq:conservation}, we infer
\begin{equation*}
\int_{s_N\leq s_1'\leq ...\leq s_N'}n_N(t-s_N,[s'-s_N]_N)d[s']_N=1 .
\end{equation*}
Consequently we have obtained the desired lower bound \eqref{eq:lowerbd}.
\\[3pt]
{\em Third step. Using the Doeblin method.}
Given $t^*> 0$ and an initial data $n_N(0)$ satisfying~\eqref{as:ID}, we use~\eqref{eq:lowerbd} to conclude that for any time $t^*$
\begin{equation}\label{eq:doeblin_lower_bound}
n_N(t^*,[s]_N)\geq a_-^Ne^{-s_Na_+}\mathbbm{1}_{s_N\leq t^*}\mathbbm{1}_{\mathcal{C}_N}.
\end{equation}
We may now apply the Doeblin method (see Appendix~\ref{sec:Doeblin_Theorem}). 

Firstly, we need to show that the total mass of $n_N(t)$ is conserved by the evolution which is the statement~\eqref{eq:conservation} (see \cite{perthameST2022} for a proof). Secondly, we need to show that there exist a positive constant $\alpha$ and a probability measure $\nu$ regardless of initial data $n_N(0)$ normalized as a probability, such that
\begin{equation*}
n_N(t^*,[s]_N)\geq \alpha \nu .
\end{equation*}
 To write the lower bound \eqref{eq:doeblin_lower_bound} as the form of $\alpha \nu$, we first calculate the total measure of this lower bound,
\begin{equation}\label{eq:total_measure}
\int_{\mathcal{C}_N}a_-^Ne^{-s_Na_+}\mathbbm{1}_{s_N\leq t^*}d[s]_N=a_-^{N} \int_{0\leq s_N\leq t^*}\frac{1}{(N-1)!}s_N^{N-1} e^{-s_Na_+}ds_N=:\alpha_N(t^*) .
\end{equation}
Since the total measure of the lower bound is $\alpha_N(t^*)$, we can write
\begin{equation}\label{control_alpha}
n_N(t^*,[s]_N)\geq \alpha_N(t^*)\frac{a_-^Ne^{-s_Na_+}\mathbbm{1}_{s_N\leq t^*}\mathbbm{1}_{\mathcal{C}_N}}{\alpha_N(t^*)},
\end{equation}
where $\frac{1}{\alpha_N(t^*)}$ is a normalizing factor and $\frac{a_-^N\exp{(-s_Na_+)}\mathbbm{1}_{s_N\leq t^*}\mathbbm{1}_{\mathcal{C}_N}}{\alpha_N(t^*)}$ is a probability density by Equation~\eqref{eq:total_measure}. Now we may apply the Doeblin method. We conclude that there is a steady state $n_N^*$ and that  for any initial data $n_N(0)$, the corresponding solution satisfies, with $\lambda_N(t^*)$ introduced in \eqref{LTB:numbers}
\begin{equation*}
\lVert n_N(t) -n_N^*\rVert_{L^1(\Rp_N)}\leq \frac{1}{1-\alpha_N(t^*)}e^{-\lambda_N(t^*) t}\lVert n_N(0)-n_N^*\rVert_{L^1(\Rp_N)},\quad \forall t\geq 0.
\end{equation*}
Here $t^*$ can be chosen as any positive constant.
\\[3pt] 
{\em Fourth step. Tightness.}
Since the steady state is also a solution of the evolution equation, we may use the result of Lemma~\ref{lemma:tightness} for $n_N^*([s]_N)$, and let $t\to \infty$. The announced result follows.
\end{proof}



\section{Hierarchy model as the strong limit $N\rightarrow +\infty$}   
\label{Sec:hierarchy}

It seems difficult to give a meaning to the boundary condition as $N \to \infty$ and to determine an appropriate notion of derivative in infinite dimensions. To avoid these difficulties, we propose to only define the $K$-marginales of the solution of the infinite-times renewal equation. Then we can write a form of BBGKY hierarchy as in classical kinetic theory, \cite{CIP1994}.

For that purpose, we assume that, for all $N$, the initial data are consistent and, in the sense of Definition~\ref{def:consistent}, there is a consistent family $n_{\infty}^{0,(K)}$ such that
\begin{equation} \label{ID:Limit}
\sup_K \|n_{N}^{(K)}(0)-n_{\infty}^{0,(K)} \|_{L^1(\Rp_K)} \to 0 \qquad \text{as} \quad N \to \infty.
\end{equation}

 Our strategy is as follows. Suppose that the limits  exist in a strong enough norm 
\begin{equation}
\label{Limit}
n_\infty^{(K)}  (t,[s]_{K})= \lim_{N \to \infty} n_{N}^{(K)}  (t,s_1,s_2,...s_{K}), \qquad K=1,\, 2,...
\end{equation}
The support property and the tightness bound in Lemma~\ref{lemma:tightness} for the $N$-times equation give
\[
n_\infty^{(K)}\geq 0 \quad \text{ is supported in }\; \mathcal{C}_K  \qquad \text{and} \qquad  \int_{\Rp_K} n_\infty^{(K)}  (t,[s]_{K}) d[s]_{K} = 1.
\]
In this limit, we obtain an infinitely coupled hierarchy system, for $K=1,\, 2,...$
\begin{equation}  \label{eqInfinity}
\left\{
\begin{matrix*}[l]
\partial_t n_\infty^{(K)}+ \displaystyle \sum_{i=1}^K \partial_{s_i} n_\infty^{(K)} +p_K([s]_K ) n_\infty^{(K)} + E_\infty^{(K)}(t,[s]_K) =0, 
\\[5pt]
n_\infty^{(K)}(t,s_1=0,s_2,...s_K)= \int_{u=0}^\infty \big[ p_K n_\infty^{(K)}+ E_\infty^{(K)} \big] (t,s_2,..., s_{K}, u)  du,
\\[5pt]
n_\infty^{(K)}(t=0,[s_K])=n_\infty^{0,{(K)}} ([s_K]).
\end{matrix*}
\right.
\end{equation}

The coupling term $E_\infty^{(K)}(t,[s]_K)$ is obtained passing formally to the limit in the expression~\eqref{couplingNK}, and reads
\begin{equation}  \label{ErInfinity}
E_\infty^{(K)}(t,[s]_K) = \sum_{i=K+1}^\infty \int_{s_{K+1}=0}^\infty \! ... \! \int_{s_{i}=0}^\infty  \varphi_i([s]_i)  n_\infty^{(i)}  (t, [s]_i) \,d[s]_{K+1, i}   ,  
\end{equation}
\[ 
\| E_\infty^{(K)}(t,\cdot) \|_{L^1(\Rp_K)} \leq  \sum_{i=K+1}^\infty \| \varphi_i \|_\infty  \| n_\infty^{(i)}\|_{L^1(\Rp^i)}.
\]
This hierarchy is our definition of the solution to the infinite-times equation, which is characterized by the following theorem.
\begin{theorem} [Strong convergence and uniqueness for the hierarchy] \label{th:uniqueness}
With Assumptions \eqref{as:RR1}, \eqref{as:ID}, \eqref{ID:Limit}, we have
\\
(i) For all $T>0$ and $K\in \N$, the sequence $\big\{n_{N}^{(K)}\big\}_N$ is a Cauchy sequence in $C\big([0,T]; L^1(\Rp_K) \big)$ and thus  it has a consistent limit $n_\infty^{(K)} \in C\big((0,\infty); L^1(\Rp_K)\big)$. 
\\
(ii) $E_N^{(K)}(t,[s]_K) \to E_\infty^{(K)}(t,[s]_K)$ in $C\big((0,T); L^1(\Rp_K) \big)$ as $N \to \infty$.
\\
(iii) $\big\{n_\infty^{(K)}\big\}_K$ is the unique consistent weak solution of the hierarchy~\eqref{eqInfinity}--\eqref{ErInfinity}. 
\\
(iv) Assuming also \eqref{as:IDbdd}, then $n_\infty^{(K)}(t,[s]_K) \leq C_K \bar n^{(K)}([s]_K)e^{\eps_K t} $ for all $t>0$.
\end{theorem}
The consistency property in (iii) allows us to define later a solution on $\Rp_\infty$, see Section~\ref{sec:MeasureSol}. Here, we again refer to weak solutions with an obvious extension of Definition~\ref{def:MSI_finite} for each $K$, see  Definition~\ref{def:hierarchy_weak_solution}. 

Notice that we also establish, see \eqref{hm:bound-marginal}, the following error estimates
\begin{align} \label{hierar:CVestimate2}
 \| n_N(t)-n_\infty^{(N)}(t) \|_{L^1(\Rp_N)} \leq \|n_{N}(0)-n_{\infty}^{(N)}(0)\|_{L^1(\Rp_N)}+ 2\eps_{N}t.
\end{align}

\begin{proof}
{\em First step. Cauchy sequence.} For $ N_1 \leq N_2$, we set $m(t):= n_{N_2}^{(N_1)}(t)-n_{N_1}(t)$. It satisfies 
\[
 \left\{
\begin{matrix*}[l]
\partial_t m  + \displaystyle \sum_{i=1}^{N_1} \partial_{s_i} m +p_{N_1}([s]_{N_1} ) m =  E_{N_2}^{(N_1)}, 
\\[5pt]
m(t,s_1=0,[s]_{2,N_1}) =  \int_{u=0}^\infty \big[ p_{N_1} m - E_{N_2}^{(N_1)} \big] (t,[s]_{2,N_1}, u)  du.
\end{matrix*}
\right.
\]
As a consequence, we also have 
\[
 \left\{
\begin{matrix*}[l]
\partial_t |m|  + \displaystyle \sum_{i=1}^{N_1} \partial_{s_i} |m| +p_{N_1}([s]_{N_1} ) |m| \leq | E_{N_2}^{(N_1)}|, 
\\[5pt]
|m|(t,s_1=0,[s]_{2,N_1})  \leq   \int_{u=0}^\infty \big[ p_{N_1} |m| + |E_{N_2}^{(N_1)}| \big] (t,[s]_{2,N_1}, u)  du.
\end{matrix*}
\right.
\]
Integrating this inequality gives
\begin{equation}\label{diff-m} 
 \frac{d}{dt} \int_{\Rp_{N_1}} |m(t,[s]_{N_1})| d[s]_{N_1} \leq 2 \int_{\Rp_{N_1}} |E_{N_2}^{({N_1})}| d[s]_{N_1} .
\end{equation}

We now estimate $| E_{N_2}^{(N_1)}|$ using~\eqref{Error_estimate} and conclude
\[
2 \int_{\Rp_{N_1}} |E_{N_2}^{({N_1})}| d[s]_{N_1} \leq \sum_{i=N_1+1}^{N_2} \| \varphi_i \|_\infty=2\eps_{N_1,N_2}.
\]
From this, recalling \eqref{diff-m} and the definition of $m$, we immediately conclude that
\begin{equation} \label{hm:bound-marginal}
\|n_{N_2}^{(N_1)}(t)-n_{N_1}(t))\|_{L^1(\Rp_{N_1})}\leq \|n_{N_2}^{(N_1)}(0)-n_{N_1}(0)\|_{L^1(\Rp_{N_1})}+ 2\eps_{N_1,N_2}t, 
\end{equation}
and marginals being controlled by the total norm, we also have
\begin{equation*} 
 \|n_{N_2}^{(K)}(t)-n_{N_1}^{(K)}(t))\|_{L^1(\Rp_{K})}\leq \|n_{N_2}^{(N_1)}(0)-n_{N_1}(0)\|_{L^1(\Rp_K)}+ 2\eps_{N_1,N_2}t, 
\end{equation*}
Thanks to Assumptions \eqref{as:RR1}, \eqref{ID:Limit}, this shows that the sequences $(n_{N}^{(K)})_N$ are Cauchy sequences and thus have limits $n_\infty^{(K)}$ as stated in (i).  Then, \eqref{hierar:CVestimate2} follows immediately from~\eqref{hm:bound-marginal}. 
\\[5pt]
%
{\em Second step. Error estimate for $E_N^{(K)}$.}
This estimate follows from Assumption \eqref{as:RR1}. Departing from~\eqref{couplingNK} we have, fixing a $J<N$ (large enough)
\begin{align*}
| E_{N}^{(K)}-E_{\infty}^{(K)}| &\leq \sum_{i=K+1}^N \| \varphi_i \|_\infty \int_{s_K\leq s_{K+1}\leq...} |n_N^{(i)}-n_\infty^{(i)} | d[s]_{K+1,i} + \sum_{i=N+1}^\infty \| \varphi_i \|_\infty \int_{s_N\leq s_{N+1}\leq...} n_\infty^{(i)} d[s]_{N+1,i}
\\
& \leq \sum_{i=K+1}^J \| \varphi_i \|_\infty \int_{s_K\leq s_{K+1}\leq...} |n_N^{(i)}-n_\infty^{(i)} | d[s]_{K+1,i}   +3 \eps_J.
\end{align*}

For all $\epsilon >0$ we may first choose $J$ so that $\eps_J \leq \epsilon$ and then, $N$ large enough so that the first term on the right hand side is also less that $\epsilon$. This proves (ii).
\\[5pt]
%
{\em Third step. Uniqueness.} To prove uniqueness is the same as showing that all the $n_\infty^{(K)}$ vanish  when for all the $n^{0,{(K)}} ([s]_K)$ vanish. Notice that here $n_\infty^{(K)}$ is the difference between two probability and thus has mass  controlled by $2$, not $1$.
\\

We have
\[
 \left\{
\begin{matrix*}[l]
\partial_t |n_\infty^{(K)}|+ \displaystyle \sum_{i=1}^K \partial_{s_i} |n_\infty^{(K)}| +p_K([s]_K ) |n_\infty^{(K)} |\leq | E_\infty^{(K)}(t,[s]_K) |, 
\\[5pt]
|n_\infty^{(K)}(t,s_1=0,[s]_{2,K})| \leq  \int_{u=0}^\infty \big[ p_K |n_\infty^{(K)}|+ |E_\infty^{(K)}| \big] (t,[s]_{2,K}, u)  du,
\\[5pt]
n_\infty^{(K)}(t=0,[s]_K)=0.
\end{matrix*}
\right.
\]
Therefore, integrating, we obtain
\[
\begin{split}
\frac{d}{dt} \int_{\Rp_K}&  |n_\infty^{(K)}| d[s]_K+ \int_{\Rp_K} p_K([s]_K ) |n_\infty^{(K)} |
 d[s]_K\\ 
\leq& \int_{\Rp_K} | E_\infty^{(K)}(t,[s]_K) | d[s]_K+ \int_{0\leq s_2\leq...\leq s_K} |n_\infty^{(K)}(t,s_1=0,[s]_{2,K})| d[s]_{2,K}, 
 \end{split}
\]
and
\[
\int_{0\leq s_2\leq ...\leq s_K} |n_\infty^{(K)}(t,s_1=0,s_2,...s_K)| d[s]_{2,K} = \int_{\Rp_K} \big[p_K([s]_K ) |n_\infty^{(K)}| + | E_\infty^{(K)}(t,[s]_K) |\big] d[s]_K ,
\]
so that we have 
\[
 \frac{d}{dt} \int_{\Rp_K}  |n_\infty^{(K)}| d[s]_K \leq 2  \int_{\Rp_K} | E_\infty^{(K)}(t,[s]_K) | d[s]_K \leq 4 \sum_{i=K+1} \| \varphi_i \|_\infty.
\] 

Finally, for all $J<K$, we conclude
\[
\int_{\Rp_J}  |n_\infty^{(J)}(t) | d[s]_J  \leq \int_{\Rp_K}  |n_\infty^{(K)}| d[s]_K \leq 4 t  \eps_K \to 0 \quad as \quad K\to  \infty.
\]
Therefore all marginales $n_\infty^{(J)}$ vanish. 

{\em Fourth step. $L^\infty$ bound.} 
This follows directly from passing to the limit in the $L^\infty$ bound for $n_K$ as stated in Proposition~\ref{pr:iniformbdd}.

The proof of Theorem \ref{th:uniqueness} is complete.
\end{proof}
\section{Uniform-in-time limit as $N \to \infty$ by the Doeblin method}
\label{sec:uniformError}

\subsection{Idea and main result}

Theorem \ref{th:uniqueness} establishes the hierarchy solution $\{n^{(K)}_{\infty}(t)\}_K$ to the infinite-times equation \eqref{eqInfinity}. Such a solution is obtained via a \textit{local-in-time} limit of the $N$-times model $n_N(t)$ as expressed in \eqref{hierar:CVestimate2}. 
Here we are interested in a stronger \textit{uniform-in-time} limit, under an additional  smallness condition on $\eps_N$. Firstly,  we aim to show that
\begin{equation}\label{uniform-approximation-tmp}
    \sup_{t\in[0,+\infty)}\|n_N(t)-n_{\infty}^{(N)}(t)\|_{L^1(\Rp_N)}\rightarrow0,\qquad \text{as $N\rightarrow\infty$}.
\end{equation}
This estimate requires a long time control which we obtain using the Doeblin condition (to be introduced in \eqref{def-doeblin-cond-sec4}).

Secondly, since we have a uniform-in-time approximation, then the stationary solution of the infinite-times equation can be approximated by that of the $N$-times equation, thus showing that 
\begin{equation} \label{uniform-approximation-tmp2}
        \|n_{N}^*-n_{\infty}^{*,(N)}\|_{L^1(\Rp_N)} \to 0, \qquad \text{as} \quad N \to \infty.
    \end{equation}

Finally, for the $N$-times equation, its long time behavior has been characterized in Section \ref{sec:StStN}. Every solution converges to the unique steady state in $L^1$ norm with an exponential rate $\lambda_N$. However, such a rate $\lambda_N$ degenerates to $0$ as $N$ goes to infinity, which prevents us to deduce the long time behavior of the infinite-times equation directly. Nevertheless, under our smallness condition on $\eps_N$, it turns out that the degeneracy of $\lambda_N$ can be compensated to get for each $N$-marginal 
\begin{equation} \label{uniform-approximation-tmp3}
        n_{\infty}^{(N)}(t)\rightarrow n_{\infty}^{*,(N)},\quad \text{in }L^1(\Rp_N)\qquad \text{as }t\rightarrow\infty.
\end{equation}

Guided by these observations, we state the following results.

\begin{theorem}\label{thm:uni-n-time-concrete}
Assume \eqref{as:ID}, \eqref{as:RR1} and \eqref{as:RR2}. Additionally, assume that the renewal rate converges fast enough in the following sense
\begin{equation}\label{cond-epsN}
    \eps_N=o\left(\frac{1}{N}\left(\frac{a_-}{a_+}\right)^{N}\right),\quad \text{as }N\rightarrow\infty,\qquad \eps_N:=\sum_{i=N+1}^{\infty}\|\phi_i\|_{\infty}.
\end{equation}
Then, the convergence results \eqref{uniform-approximation-tmp}--\eqref{uniform-approximation-tmp3} hold and the steady state infinite hierarchy $n_\infty^{*,(N)}$ is unique.
\end{theorem}

To prove this theorem, we first give a general framework in subsection~\ref{subsec:uni-time-abstract}, which motivates results on the uniform-in-time limit and its consequences with an abstract smallness condition in subsection~\ref{subsec:N-long-time-revisited}. In fact Theorem~\ref{thm:uni-n-time-concrete} is a specific version of the more general and precise Theorem~\ref{thm:uni-n-time}. To get the specific condition \eqref{cond-epsN}, in subsection~\ref{subsec:strong_condition} we study carefully the constants in subsection~\ref{sec:StStN}.


\subsection{A framework towards uniform-in-time approximation}
\label{subsec:uni-time-abstract}

We start with a local-in-time control on the $K$ marginal of the solution $n_N$ of \eqref{eqKTR} with a consistent family of initial probability densities. From the proof of \eqref{hm:bound-marginal} in Section~\ref{Sec:hierarchy}, with $N_1=K,N_2=N$, we infer the control for all $t\geq 0$ and $\tau \geq 0$, 
\begin{equation} \label{uit:bound-marginal}
    \|n_{K}(t+\tau)-n_{N}^{(K)}(t+\tau )\|_{L^1(\Rp_K)}\leq \|n_{K}(\tau)-n_{N}^{(K)}(\tau)\|_{L^1(\Rp_K)}+ 2\eps_{K,N}t,
\end{equation}

Now we try to utilize the Doeblin condition to give a uniform-in-time bound. We say the $K$-times problem satisfies the Doeblin condition if there exists $t_K^*>0$, $\alpha_K\in(0,1)$ and a probability density $\nu_K$ such that
\begin{equation}\label{def-doeblin-cond-sec4}
    n_{K}(t_K^*)\geq \alpha_K \nu_K,
\end{equation} 
independently of the initial probability density $n_{K}(0)$. 

\begin{remark}
 In Section \ref{sec:StStN}, the Doeblin condition \eqref{def-doeblin-cond-sec4} is established under Assumption~\eqref{as:RR2}, with concrete expressions on the constants $\alpha$ and $t^*$. In this subsection, we do not assume \eqref{as:RR2} but work with the more abstract condition~\eqref{def-doeblin-cond-sec4}. We establish concretely this condition using Assumption \eqref{as:RR2} in subsection~\ref{subsec:strong_condition}.
\end{remark}

\begin{proposition}\label{prop:uni-in-time-1}
Given $K$, suppose the $K$-times problem satisfies the Doeblin condition \eqref{def-doeblin-cond-sec4}. Then with Assumptions \eqref{as:RR1} and \eqref{def-eps}, we have the uniform estimates
\begin{equation}\label{eq:uni-time-1}
\limsup_{k \rightarrow\infty}\|n_{K}(kt_{K}^*)-n_{N}^{(K)}(kt_{K}^*)\|_{L^1(\Rp_K)}\leq 2\frac{t_{K}^*\eps_{K,N}}{\alpha_{K}},
\end{equation}
\begin{equation}\label{eq:uni-time-str}
        \limsup_{t\rightarrow+\infty}\|n_{K}(t)-n_{N}^{(K)}(t)\|_{L^1(\Rp_K)}\leq 4\frac{t_{K}^*\eps_{K,N}}{\alpha_{K}}. 
\end{equation}
\end{proposition}
\begin{proof}[Proof of Proposition \ref{prop:uni-in-time-1}] 
Firstly, we prove that, for all $\tau \geq 0$, 
\begin{equation}\label{eq:uni-time-tmp-key}
  \|n_{K}(t^*_K+\tau)-n_{N}^{(K)}(t^*_K+\tau)\|_{L^1(\Rp^K)}  \leq (1-\alpha_K)\|n_{K}(\tau)-n_{N}^{(K)}(\tau)\|_{L^1(\Rp_K)}+2\eps_{K,N}t^*_K.
\end{equation}
To do so, we denote by $S^{(K)}_{t}$ the semi-group associated with the $K$-times problem. Using the Doeblin Theorem, we have
\begin{equation*}
    \|n_{K}(t^*_K+\tau)-S^{(K)}_{t^*_K}\bigl(n_{N}^{(K)}(\tau)\bigr)\|_{L^1(\Rp_K)}\leq (1-\alpha_K)\|n_{K}(\tau)-n_{N}^{(K)}(\tau)\|_{L^1(\Rp_K)}.
\end{equation*} 
Then, applying the control \eqref{uit:bound-marginal} to $S^{(K)}_{t^*_K}\bigl(n_{N}^{(K)}(\tau)\bigr)$ and $n_{N}(t^*_K+\tau)$, we have
\begin{equation*}
    \|S^{(K)}_{t^*_K}\bigl(n_{N}^{(K)}(\tau)\bigr)-n_{N}^{(K)}(t^*_K+\tau)\|_{L^1(\Rp_K)}\leq\|n_{N}^{(K)}(\tau)-n_{N}^{(K)}(\tau)\|_{L^1(\Rp_K)}+2\eps_{K,N}t^*_K= 2\eps_{K,N}t^*_K.
\end{equation*}
We combine the above two estimates via the triangle inequality to obtain
\begin{align*}
  \|n_{K}(t^*_K+\tau)-n_{N}^{(K)}(t^*_K+\tau)\|_{L^1(\Rp_K)}\leq (1-\alpha_K)\|n_{K}(\tau)-n_{N}^{(K)}(\tau)\|_{L^1(\Rp_K)}+2\eps_{K,N}t^*_K.
\end{align*}

Secondly, we iterate \eqref{eq:uni-time-tmp-key} to obtain 
\begin{equation*}
    \|n_{K}(kt^*+\tau)-n_{N}^{(K)}(kt^*+\tau)\|_{L^1(\Rp_K)}\leq (1-\alpha_K)^k \|n_{K}(\tau)-n_{N}^{(K)}(\tau)\|_{L^1(\Rp_K)}+2t^*\eps_{K,N}\left(\sum_{i=0}^{k-1}(1-\alpha_K)^i\right)
\end{equation*} 
and,   this gives 
\begin{equation}\label{eq:uni-time-1-2}
\|n_{K}(kt^*+\tau)-n_{N}^{(K)}(kt^*+\tau)\|_{L^1(\Rp_K)}\leq (1-\alpha_K)^k \|n_{K}(\tau)-n_{N}^{(K)}(\tau)\|_{L^1(\Rp_K)}+ 2\frac{t_{K}^*\eps_{K,N}}{\alpha_{K}}.
\end{equation}
Choosing $\tau=0$, the inequality \eqref{eq:uni-time-str} follows immediately.
\\

Thirdly, for $\tau \in (0, t^*_K)$ we use the control \eqref{hm:bound-marginal} to obtain
\begin{equation*}
\begin{aligned}
 \|n_{K}(\tau)-n_{N}^{(K)}(\tau)\|_{L^1(\Rp_K)}
 &\leq \|n_{K}(0)-n_{N}^{(K)}(0)\|_{L^1(\Rp_K)}+2\tau\eps_{K,N}
 \\ &\leq   \|n_{K}(0)-n_{N}^{(K)}(0)\|_{L^1(\Rp_K)}+2\frac{t^*_K\eps_{K,N}}{\alpha_K}.
\end{aligned}
\end{equation*}
Combined with \eqref{eq:uni-time-1-2}, it gives
\begin{equation}\label{eq:uni-time-1-3}
    \|n_{K}(kt^*_K+\tau)-n_{N}^{(K)}(kt^*_K+\tau)\|_{L^1(\Rp^K)}\leq (1-\alpha_K)^k \|n_{K}(0)-n_{N}^{(K)}(0)\|_{L^1(\Rp_K)}+ 4\frac{t_{K}^*\eps_{K,N}}{\alpha_{K}}.
\end{equation}
And then \eqref{eq:uni-time-str} follows.
\end{proof}

A steady state $n^*_K$ for the $K$ problem can be built via the Doeblin method as in Section~\ref{sec:StStN}. By Proposition \ref{prop:uni-in-time-1}, we immediately deduce an estimate between the steady states.
\begin{corollary}[Distance between the steady states]\label{cor:dis-steady} Under the same condition as in Proposition \ref{prop:uni-in-time-1}, assume for $K<N$ the steady states for the $N$-times and the $K$-times problem exist. Then we have 
\begin{equation}
    \|n_{K}^*-n_{N}^{*,(K)}\|_{L^1(\Rp^K)}\leq 2\frac{t_{K}^*\eps_{K,N}}{\alpha_{K}}.
\end{equation}
\end{corollary}

Note that the bounds in Proposition \ref{prop:uni-in-time-1}, albeit uniform-in-time, might not be meaningful as $N\rightarrow\infty$. This motivates us to look into the quantity 
\begin{equation*}
    2\frac{t_{K}^*\eps_{K,N}}{\alpha_{K}}\leq 2\frac{t_{K}^*\eps_{K}}{\alpha_{K}},
\end{equation*}
where $\eps_K$ is defined in \eqref{def-eps}. A uniform-in-time  limit as $N\rightarrow\infty$ would be obtained if we can choose $\alpha_N,t^*_N$ such that 
\begin{equation}\label{uni-condition}
   2\frac{t_{N}^*\eps_{N}}{\alpha_{N}}\rightarrow0,\qquad \text{as }N\rightarrow\infty.
\end{equation}
Here the decay property of the renewal rate plays a central role. As we discuss it later, generically ${t^*_N}/{\alpha_N}$ diverges to infinity as $N\rightarrow\infty$. Thus we need that $\eps_N$ decays fast enough to compensate the divergence of ${t^*_N}/{\alpha_N}$. 

\subsection{Uniform-in-time limit and consequences} 
\label{subsec:N-long-time-revisited}

Before giving concrete conditions to ensure \eqref{uni-condition}, we first improve Theorem~\ref{th:uniqueness} when \eqref{uni-condition} is satisfied.



\begin{theorem}\label{thm:uni-n-time} Assume \eqref{as:RR1} and that for each $K$ the Doeblin condition \eqref{def-doeblin-cond-sec4} is satisfied with $\alpha_K,t^*_K>0$. If additionally \eqref{uni-condition} is true, then we have
\\[2pt]
(i) (Uniform-in-time limit) 
      Assume the initial data satisfy \eqref{ID:Limit}.
      Then for each $K$, $n_N^{(K)}(t)$ is a Cauchy sequence in $C_b([0,+\infty),L^1(\Rp^K))$, whose limit $n_{\infty}^{(K)}(t)$ is the hierarchy solution to the infinite-times equation \eqref{eqInfinity}. Moreover, we have a uniform-in-time approximation 
    \begin{equation}\label{uni-time-limit-thm-statement}
            \sup_{t\in[0,+\infty)}\|n_N(t)-n_{\infty}^{(N)}(t)\|_{L^1(\Rp^N)}\rightarrow0,\qquad \text{as $N\rightarrow\infty$}.
    \end{equation}
(ii) (Limit of the steady state) 
    As $N$ goes to infinity, $n_N^{*,(K)}$ has a unique (strong) limit $n_{\infty}^{*,(K)}$, which is a steady state of the infinite-times hierarchy and we have the error estimate
    \begin{equation}\label{uni-steady-statement}
        \|n_{N}^*-n_{\infty}^{*,(N)}\|_{L^1(\Rp_N)}\leq  2\frac{t_{N}^*\eps_{N}}{\alpha_{N}} \to 0, \qquad \text{as} \quad N \to \infty.
    \end{equation}
(iii) (Long term behavior for the infinite-times equation) For a general solution of the infinite-times equation $n_{\infty}(t)$, its every marginal converges in time to the marginal of $n_{\infty}^*$. For every $N$, we have
    \begin{equation}
        n_{\infty}^{(N)}(t)\rightarrow n_{\infty}^{*,(N)},\quad \text{in }L^1(\Rp_N)\qquad \text{as }t\rightarrow\infty.
    \end{equation}
(iv) (Uniqueness)  The steady state solution $n_{\infty}^{*, (N)}$ of the infinite-times equation, as defined above, is unique.
\end{theorem}

\begin{proof}
{\em First part. Uniform-in-time limit.} Fixed $K\in\mathbb{N}$. Note that the $L^1$ norm of a function can control the $L^1$ norm of its marginal. Precisely, for $K\leq N_1\leq N_2$, we have
\begin{equation*}
    \|n_{N_1}^{(K)}(t)-n_{N_2}^{(K)}(t)\|_{L^1(\Rp_K)}\leq \|n_{N_1}^{(N_1)}(t)-n_{N_2}^{(N_1)}(t)\|_{L^1(\Rp_{N_1})}=\|n_{N_1}(t)-n_{N_2}^{(N_1)}(t)\|_{L^1(\Rp_{N_1})},
\end{equation*} for each $t\geq0$. Thus we get
\begin{equation}\label{pf-tmp-step-1-uni-time}
    \sup_{t\in[0,+\infty)}\|n_{N_1}^{(K)}(t)-n_{N_2}^{(K)}(t)\|_{L^1(\Rp_{K})}\leq\sup_{t\in[0,+\infty)}\|n_{N_1}(t)-n_{N_2}^{(N_1)}(t)\|_{L^1(\Rp_{N_1})}.
\end{equation}Thanks to the uniform-in-time bound \eqref{eq:uni-time-1-3}, we obtain
\begin{equation*}
    \sup_{t\in[0,+\infty)}\|n_{N_1}(t)-n_{N_2}^{(N_1)}(t)\|_{L^1(\Rp_{N_1})}\leq \|n_{N_1}(0)-n_{N_2}^{(N_1)}(0)\|_{L^1(\Rp_{N_1})}+4\frac{t_{N_1}^*\eps_{N_1}}{\alpha_{N_1}},
\end{equation*} combining which with \eqref{pf-tmp-step-1-uni-time} we deduce for $K\leq N_1\leq N_2$
\begin{align}
    \sup_{t\in[0,+\infty)}\|n_{N_1}^{(K)}(t)-n_{N_2}^{(K)}(t)\|_{L^1(\Rp_{K})}&\leq \|n_{N_1}(0)-n_{N_2}^{(N_1)}(0)\|_{L^1(\Rp_{N_1})}+4\frac{t_{N_1}^*\eps_{N_1}}{\alpha_{N_1}}.\label{pf-tmp-step-1-uni-time-2}
\end{align} 
As $N_1 \to \infty$,  the first term goes to zero due to the assumption on initial data \eqref{ID:Limit}, and the second term vanishes thanks to \eqref{uni-condition}, hence, \eqref{pf-tmp-step-1-uni-time-2} shows that for fixed $K$, $\{n_N^{(K)}(t)\}_{N}$ is a Cauchy sequence in $C_b\big([0,+\infty),L^1(\Rp_K)\big)$, which has a limit $n_{\infty}^{(K)}(t)$. As in  Theorem~\ref{th:uniqueness}, we can check that the limit $n_{\infty}^{(K)}$ is a consistent solution to the hierarchy problem~\eqref{eqInfinity}--\eqref{ErInfinity}.

Moreover, we can obtain a more concrete estimate. Fix $K= N_1=N$ in \eqref{pf-tmp-step-1-uni-time-2} and let $N_2$ goes to infinity, and we deduce
\begin{equation}\label{uni-time-limit-thm-statement-concrete}
     \sup_{t\in[0,+\infty)}\|n_{N}(t)-n_{\infty}^{(N)}(t)\|_{L^1(\Rp_{K})}\leq \|n_{N}(0)-n_{\infty}^{0,(N)}\|_{L^1(\Rp_{N})}+4\frac{t_{N}^*\eps_{N}}{\alpha_{N}},
\end{equation} 
and \eqref{uni-time-limit-thm-statement} follows  thanks to \eqref{ID:Limit} and \eqref{uni-condition},.

{\em Second part. Limit of the steady state.} 
The existence and uniqueness of $n^*_N$ are ensured by the Doeblin theorem. For $K\leq N_1\leq N_2$ we have immediately 
\begin{align}\label{step2-thm-uni-tmp-3}
    \|n_{N_1}^{*,(K)}-n_{N_2}^{*,(K)}\|_{L^1(\Rp_{K})}\leq \|n_{N_1}^{*}-n_{N_2}^{*,(N_1)}\|_{L^1(\Rp_{N_1})} \leq 2\frac{t_{N_1}^*\eps_{N_1}}{\alpha_{N_1}},
\end{align}
as in Corollary \ref{cor:dis-steady}, that is passing to the limit in \eqref{eq:uni-time-1}. 

Thanks to \eqref{uni-condition}, \eqref{step2-thm-uni-tmp-3} implies that for fixed $K$, $\{n_{N}^{*,(K)}\}_{N}$ is a Cauchy sequence in $L^1(\Rp^K)$. Denote its limit as $n_{\infty}^{*,(K)}$. Then we fix $K,N_1$ and let $N_2$ goes to infinity in \eqref{step2-thm-uni-tmp-3} to derive a concrete error estimate
\begin{equation*}
    \|n_{N_1}^{*,(K)}-n_{\infty}^{*,(K)}\|_{L^1(\Rp_{K})}\leq2\frac{t_{N_1}^*\eps_{N_1}}{\alpha_{N_1}},
\end{equation*}which goes to zero as $N_1$ goes to infinity by \eqref{uni-condition}. Taking $K=N_1$ we deduce \eqref{uni-steady-statement}. 

Similar to Theorem \ref{th:uniqueness}, we can see that $n_{\infty}^{*,(K)}$ is a steady solution to the infinite-times hierarchy problem via passing the limit in the weak formulation. 

{\em Third part. Long term behavior for the infinite-times equation.} 
For a solution to the infinite problem $n_{\infty}(t)$, we can construct the solution $n_N(t)$ to the $N$ problem  with initial data $n_\infty^{0,(N)}$. 
Using the triangle inequality, we have for $N\geq K$, 
\begin{align*}
    \|n_{\infty}^{(K)}(t)-&n_{\infty}^{*,(K)}\|_{L^1(\Rp_{K})}
    \leq\|n_{\infty}^{(N)}(t)-n_{\infty}^{*,(N)}\|_{L^1(\Rp_{N})}
    \\&\leq \|n_{\infty}^{(N)}(t)-n_{N}(t)\|_{L^1(\Rp_{N})}+\|n_{N}(t)-n_{N}^*\|_{L^1(\Rp_{N})}+\|n_{N}^*-n_{\infty}^{*,(N)}\|_{L^1(\Rp_{N})}
    \\&\leq  4\frac{t_{N}^*\eps_{N}}{\alpha_{N}} +\|n_{N}(t)-n_{N}^*\|_{L^1(\Rp_{N})}+  2\frac{t_{N}^*\eps_{N}}{\alpha_{N}},
\end{align*}
where we have used \eqref{uni-time-limit-thm-statement-concrete} and  \eqref{uni-steady-statement}. For $N$ large enough, the two error terms can be made as small as we wish using Assumption~\eqref{uni-condition} .
Finally, applying the Doeblin theorem to the $N$ problem, we know that $\|n_{N}(t)-n_{N}^*\|_{L^1(\Rp_{N})}\to 0$ for all $t$ large enough. Therefore the proof is completed.

{\em Fourth part.} In part (iii) we have just established that every solution converges to $n^*_{\infty}$ in the long time. Hence, $n^*_{\infty}$ is the unique steady state for the infinite-times problem.
\end{proof}

%
\subsection{Towards a uniform Doeblin condition}
\label{subsec:strong_condition}

Theorem \ref{thm:uni-n-time} differs from Theorem \ref{thm:uni-n-time-concrete} in its abstract Doeblin assumption \eqref{uni-condition}. We now give a concrete conditions to ensure \eqref{uni-condition}.

\begin{proposition}\label{prop:bound-doeblin-rate}
Assume \eqref{as:RR2}, then we can choose $t^*_N>0$ and $\alpha_N\in(0,1)$ such that the  condition \eqref{def-doeblin-cond-sec4} holds with
\begin{equation}\label{doeblin-constant-bound}
    \frac{\alpha_N}{t^*_N}\geq \frac{a_-}{2N}\left(\frac{a_-}{a_+}\right)^{N-1}.
\end{equation}
\end{proposition}

Combining Theorem \ref{thm:uni-n-time} with Proposition \ref{prop:bound-doeblin-rate}, we get Theorem \ref{thm:uni-n-time-concrete}.
\begin{proof}[Proof of Theorem \ref{thm:uni-n-time-concrete}]
Choose $\alpha_N$ and $t^*_N$ satisfying \eqref{doeblin-constant-bound} by Proposition \ref{prop:bound-doeblin-rate}. Then thanks to Assumption~\eqref{cond-epsN} on the decay of~$\eps_N$, the condition \eqref{uni-condition} holds. Hence, we can apply Theorem \ref{thm:uni-n-time} to derive all the conclusions of Theorem ~\ref{thm:uni-n-time-concrete}.
\end{proof}

\begin{proof}[Proof of Proposition \ref{prop:bound-doeblin-rate}]

By Theorem \ref{thm:ststN}, Equation \eqref{control_alpha} and Definition \eqref{LTB:numbers}, for any $t^*>0$, we can choose $\alpha=\alpha_N(t^*)>0$ satisfying the Doeblin condition \eqref{def-doeblin-cond-sec4} with
\begin{equation}\label{doeblin-constant-n-recall}
   \frac{\alpha_N(t^*)}{t^*}=\frac{(a_-)^N}{t^*}\int_0^{t^*}\frac{1}{(N-1)!}s^{N-1}e^{-sa_+}ds.
\end{equation} 

Now we want to find a suitable $t^*$ such that the right hand side of \eqref{doeblin-constant-n-recall} is relatively large and easy to compute. To do so, we rewrite \eqref{doeblin-constant-n-recall} via a change of variable in the integral
\begin{align*}
    \frac{\alpha_N(t^*)}{t^*}
    &=\frac{(a_-)^N}{a_+t^*}\frac{1}{(a_+)^{N-1}}\int_0^{t^*}\frac{1}{(N-1)!}(a_+s)^{N-1}e^{-sa_+}d(a_+s)\\
    &=a_-\left(\frac{a_-}{a_+}\right)^{N-1}\frac{1}{\tau}\int_0^{\tau}\frac{1}{(N-1)!}\tilde{s}^{N-1}e^{-\tilde{s}}d\tilde{s},
\end{align*} 
with $\tau=a_+t^*$. Then we choose $t^*_N=N/a_+$, i.e., $\tau=N$, to get
\begin{align*}
    \frac{\alpha_N(t^*_N)}{t^*_N}&=a_-\left(\frac{a_-}{a_+}\right)^{N-1}\frac{1}{N}\int_0^{N}\frac{1}{(N-1)!}\tilde{s}^{N-1}e^{-\tilde{s}}d\tilde{s}\geq  \frac{a_-}{2N}\left(\frac{a_-}{a_+}\right)^{N-1},
\end{align*} where we use the following inequality for incomplete gamma functions (c.f. survey \cite[Section 5.1, Equation (5.6)]{gautschi1998incomplete} and  \cite{vietoris1983dritter})
\begin{equation*}\label{gamma-inequality}
    \int_0^{N}\frac{1}{(N-1)!}\tilde{s}^{N-1}e^{-\tilde{s}}d\tilde{s}\geq \frac{1}{2}.
\end{equation*}
With these choices of $t^*_N$ and $\alpha_N(t^*_N)$, Proposition \ref{prop:bound-doeblin-rate} is proved.
\end{proof}

The quantity $\alpha_N/t^*_N$ is closely related to the exponential convergence rate $\lambda_N$ of the $N$-times model given in Theorem \ref{thm:ststN}. Indeed, by \eqref{LTB:numbers}, we have for $\alpha_N$ small
\begin{equation*}
    \lambda_N=-\frac{\ln(1-\alpha_N)}{t^*_N}\approx \frac{\alpha_N}{t^*_N}.
\end{equation*}
With this in mind, we may interpret \eqref{uni-condition} as using the fast decay of $\eps_N$ to compensate the degeneracy of $\lambda_N$. And Proposition \ref{prop:bound-doeblin-rate} gives some hints on how $\lambda_N$ deteriorates.

Such a degeneracy is essential for the convergence rate in the strong norm. Intuitively, for the $N$-times model, one needs to spike for at least $N$ times to forget the initial dependence, which explains the factor $\frac{1}{N}$ in \eqref{doeblin-constant-bound}. The other factor $\left(\frac{a_-}{a_+}\right)^{N-1}$, which dominates as it is exponential, might look less straightforward, which may reflect the effect of the heterogeneity of the renewal rate. 

\section{Measure solution on $\Rp_\infty$ as the weak limit $N\rightarrow +\infty$}
\label{sec:MeasureSol}

So far, the infinite-times renewal equation is understood as the hierarchy~\eqref{eqInfinity}, and well-posedness is established in Theorem \ref{th:uniqueness} for $L^1$ initial data. We now generalize the notion of solution to measures, define a notion of weak formulation and prove its well-posedness. 

Recall that the hierarchy solution studied in Section \ref{Sec:hierarchy} satisfies the \textit{consistency condition},  for each $t\geq0$,
\begin{equation*}
n_{\infty}^{(K)}(t,[s]_K)=\int_{s_K}^\infty n_{\infty}^{(K+1)}(t,[s]_{K+1})ds_{K+1}.
\end{equation*}
Then, the Kolmogorov extension theorem, see Appendix~\ref{ap:K}, allows us to build an infinite-dimensional measure $n_{\infty}(t)\in \mathcal{P}(\mathcal{C}_{\infty})$ for each $t\geq0$, such that for each $K$ the measure $n_\infty^{(K)}(t)$ is exactly the $K$ marginal of $n_{\infty}(t)$ on $\mathcal{C}_{K}$. In this way, we give a precise meaning to the infinite-dimensional object $n_{\infty}(t)$. As far as topology is concerned, we use the following weak topology.  
\begin{definition}[Weak topology and weakly continuous function valued in $\mathcal{M}(\Rp_\infty)$]\label{def:weak_continuity_infinite}
\ \\
(i) Given a sequence of elements $\{f_{j}\}_j$ in $\mathcal{M}(\Rp_\infty)$, we say it weakly converges to an element $f\in \mathcal{M}(\Rp_\infty)$, if for each $K\geq 1$, the $K$ marginals $\{f_j^{(K)}\}_j$ weakly converges to $f^{(K)}$ in $\mathcal{M}(\Rp_K)$.\\
(ii) Based on this weak topology, we say that $f(t,[s]_{\infty})$ is an element in $C_w\big([0,+\infty);\mathcal{M}(\Rp_\infty)\big)$ if for each $K\geq 1$ we have $f^{(K)}\in C_w\big([0,+\infty);\mathcal{M}(\Rp_K)\big)$.  
\end{definition}

We then define a notion of weak measure solution (see Definition \ref{def:MSI}) for the infinite-times renewal equation \eqref{eq:infinite}, and $n_{\infty}(t)$ defined above is a weak measure solution. It is equivalent to the solution of the Hierarchy system built in Section~\ref{Sec:hierarchy}, see Lemma~\ref{lemma:equivalence}. Such a definition offers another viewpoint to treat the hierarchy system as a unified object in infinite dimension, which facilitates the convergence proof in the Monge-Kantorovich distance, see  Section~\ref{sec:infinite_MK_convergence}.

\subsection{Weak solution of infinite-times renewal equation on $\Rp_\infty$} 
\label{sec:WFM}  
We first recall the definition of weak solutions for hierarchy model \eqref{eqInfinity}--\eqref{ErInfinity}.
\begin{definition}[Weak hierarchy solution] \label{def:hierarchy_weak_solution}
Given $n_{\infty}^{(K)}(t)\in C_w\big([0,+\infty);\mathcal{P}(\Rp_K)\big)$ for all $K\geq 1$, a consistent family,  $\{n_{\infty}^{(K)}\}_K$ is a weak solution of Equation~\eqref{eqInfinity}, if for all positive $T>0$, positive integer $K$ and all test functions $\psi\in C_b^{1}\big([0,T]\times\Rp_K\big)$, we have,
\begin{align}
-\int_0^T \! \! \int_{\Rp_K}
& n_{\infty}^{(K)}(t,d[s]_{K})\Big(\p_t\psi(t,[s]_{K})+\sum_{i=1}^{K}\p_{s_i}\psi(t,[s]_{K})+p_K([s]_K)\big(\psi(\tau[s]_K)-\psi([s]_K)\big)\Big)dt \notag\\
=&\int_0^T \! \! \int_{\Rp_K}E_{\infty}^{(K)}(t,d[s]_K)\big(\psi(t,\tau[s]_K)-\psi(t,[s]_K)\big)d[s]_Kdt
\label{eq:hierarchy_solution_1} \\
 &+ \int_{\Rp_K}\psi(0,[s]_K)n_{\infty}^{(K)}(0,d[s]_K)-\int_{\Rp_K}\psi(T,[s]_K)n_{\infty}^{(K)}(T,d[s]_K). \notag
\end{align}
\end{definition}

Recalling Definition \ref{ErInfinity}, the term $E_{\infty}^{(K)}$ is now written,
\begin{equation*}
E_{\infty}^{(K)}(t,d[s]_K)=\sum_{i=K+1}^{+\infty}\int_{s_K\leq s_{K+1}\leq ...\leq s_i}n_{\infty}^{(i)}(t,[s]_K,d[s]_{K+1,i})\varphi_i([s]_i),
\end{equation*}
and by Assumption~\eqref{as:RR1}, we have the following decomposition,
\begin{equation*}
p_K([s]_K)+\sum_{i=K+1}^{\infty}\varphi_i([s]_i)=p_{\infty}([s]_{\infty}).
\end{equation*}
Using these relations, we can write Equation~\eqref{eq:hierarchy_solution_1} in terms of $n_\infty$. This leads us to define the following equivalent notion of weak solution in $C_w\big([0,+\infty);\mathcal{M}(\mathcal{C}_{\infty})\big)$, where we can also define analogously the weak solution of finite-times equation with measure data.
\begin{definition}[Weak solution of the infinite-times renewal equation]
\label{def:MSI} 
Given $n_{\infty}(0)\in \mathcal{P}(\mathcal{C}_{\infty})$, an element $n_{\infty}\in C_w\big([0,+\infty);\mathcal{P}(\Rp_\infty)\big)$ is a weak solution of the infinite-times renewal Equation~\eqref{eq:infinite}, if for arbitrary $T\geq 0$, integer $K\geq 1$ and arbitrary test function $\psi\in C_b^1\big([0,T]\times \Rp_K\big)$ we have,
\begin{equation}\label{eq:measure-solution}
\begin{split}
 -\int_0^T \! \! \int_{\Rp_\infty} &n_{\infty}(t,d[s]_{\infty})\Big(\p_t\psi(t,[s]_{K})+\sum_{i=1}^{K}\p_{s_i}\psi(t,[s]_{K})+p_{\infty}([s]_\infty)\big(\psi(t,\tau[s]_K)-\psi(t,[s]_K)\big)\Big)dt\\
&=\int_{\Rp_\infty}\psi(0,[s]_K)n_{\infty}(0,d[s]_\infty)-\int_{\Rp_\infty}\psi(T,[s]_K)n_{\infty}(T,d[s]_\infty) .
\end{split}
\end{equation}
\end{definition}
 
\begin{lemma}[{Equivalence Lemma}]\label{lemma:equivalence}
Assume the renewal rate satisfies \eqref{as:RR1}. Let $n_{\infty}\in C_w\big([0,+\infty);\mathcal{M}(\Rp_\infty)\big)$ a weak solution of the infinite-times renewal Equation~\eqref{eq:infinite}, then the family of marginals $\{n_{\infty}^{(K)}\}_K$ is a weak solution of the hierarchy model~\eqref{eqInfinity}; Let $\{n_{\infty}^{(K)}\}_{K}$ a weak solution of the hierarchy model, then its Kolmogorov extension $n_{\infty}$ is a weak solution of the infinite-times renewal equation. 
\end{lemma}
This lemma is an obvious consequence of the equivalence between Equations \eqref{eq:hierarchy_solution_1} and \eqref{eq:measure-solution} in the definitions of weak solutions.

For the hierarchy solution with $L^1$ initial data, we have obtained well-posedness via the limit $N\rightarrow\infty$ from $N$-times model in Theorem~\ref{th:uniqueness}. This result can be readily extended to the case with measure initial data. Based on the equivalence above, this also means the infinite-times renewal equation~\eqref{eq:infinite} has a unique weak solution in $C_w\big([0,+\infty);\mathcal{P}(\Rp_\infty)\big)$, treated as an infinite dimensional measure as in Definition \ref{def:MSI}.

\begin{theorem}[Well-posedness of the measure solution]\label{th:wellposed_hierarchy}
Assume \eqref{as:RR1} and consider a consistent hierarchy of initial distributions $\{n_{\infty}^{0,(K)}\}_{K}$, where, for all $K$,  $\Big(n_{\infty}^{0,(K+1)}\Big)^{(K)}=n_{\infty}^{0,(K)}\in\mathcal{P}(\mathcal{C}_K)$. Then, we have
\\
(i) The hierarchy system \eqref{eqInfinity}-\eqref{ErInfinity} has a unique weak solution $\{n_{\infty}^{(K)}(t)\}_K$ in the sense of Definition~\ref{def:hierarchy_weak_solution}.
\\
(ii) Equivalently, the infinite-times equation \eqref{eq:infinite} has a unique weak solution $n_{\infty}(t)$ in the sense of Definition~\ref{def:MSI}. Here the initial data $n_{\infty}^0\in\mathcal{P}(\Rp_{\infty})$ is the Kolomorgov extension of $\{n_{\infty}^{0,(K)}\}_{K}$.
\\
(iii) Let $n_N(t)$ be the solution of the $N$-times equation, with renewal rate $p_{N}$ as \eqref{as:RR1} and initial data $n_{\infty}^{0,(N)}\in\mathcal{P}(\Rp_N)$. Then for every marginal we have local-in-time strong convergence. More precisely, for every fixed $T>0$ and $K\in\mathbb{N}^+$, we have, we the notations of Section~\ref{sec:notations},
\begin{equation*}
    \sup_{0\leq t\leq T}\|n_N^{(K)}(t)-n_{\infty}^{(K)}(t)\|_{\mathcal{M}^1(\Rp_K)}\rightarrow 0,\qquad\qquad\text{as $N\rightarrow\infty$.}
\end{equation*}
\end{theorem}

The proof for part (i) and (iii) of Theorem~\ref{th:wellposed_hierarchy} is identical to that of Theorem~\ref{th:uniqueness}, except that the $L^1$ norm is replaced by the total variation norm for measures. Part (ii) follows from the equivalence in Lemma~\ref{lemma:equivalence}. 

\subsection{Weak limit $N\to \infty$ via tightness}

A more direct proof of the limit $N\to \infty$ in the finite-times equation can be obtained by the method of 
tightness, which can also be used for other purposes (e.g. Proposition \ref{prop:stationary-infinite}). Here the limit holds  in the weak sense of measures, which is natural in the measure setting. With correct assumptions on the initial data, this result is a consequence of the strong convergence for marginals in Theorem~\ref{th:wellposed_hierarchy}. 

\begin{theorem}[Weak Limit]\label{th:weak_limit_of_finite}
Assume \eqref{as:RR1} and \eqref{as:RR2} and give an initial data $n_{\infty}(0)\in\mathcal{P}(\mathcal{C}_\infty)$ such that
\begin{equation}\label{initial-tight}
    \int_{\Rp_{\infty}} \sum_{i=1}^{+\infty}\frac{|s_i|}{2^i}\; n_{\infty}(0,d[s]_{\infty})<+\infty.
\end{equation}
Let $n_{\infty}(t)$ be the unique measure solution to the infinite-times equation in the sense of Definition \ref{def:MSI}. Let $n_{N}(t)$ be the corresponding solutions to the $N$-times equation with initial data $n_{\infty}^{(N)}(0)\in\mathcal{P}(\mathcal{C}_N)$. Then we have the following weak limits as $N\rightarrow\infty$.

(i) For all $t\geq 0$, integer $K\geq 1$ and test function $\phi\in C_b^0(\mathcal{C}_K)$,
\begin{equation}\label{eq:weak_convergence_lem_2}
\int_{\mathcal{C}_N}n_N(t,d[s]_N)\phi([s]_K)\rightarrow \int_{\mathcal{C}_\infty}n_{\infty}(t,d[s]_{\infty})\phi([s]_K),\qquad \text{as $N\rightarrow +\infty$}.
\end{equation}

(ii) For all $T\geq 0$, integer $K \geq 1$ and test function $\psi\in C_b^0\big([0,T]\times\mathcal{C}_K\big)$,
\begin{equation}\label{eq:weak_convergence_lem}
\int_0^T \! \! \int_{\Rp_N} n_{N}(t,d[s]_{N})\psi(t,[s]_{K})dt\rightarrow \int_0^T \! \! \int_{\Rp_\infty} n_{\infty}(t,d[s]_{\infty})\psi(t,[s]_{K})dt,\qquad \text{as $N\rightarrow +\infty$.}
\end{equation}
\end{theorem}

A first step towards this result is to prove the tightness of the family $\{n_N\}_N$.
\begin{lemma}\label{lm:tightness-R^N} Given a sequence of probability measures $\{n_N\}_{N\geq1}$, , $n_N\in \mathcal{P}(\Rp_N)$ with the uniform bound,
\begin{equation}\label{uniform-tight-bound}
\int_{\Rp_{N}} \sum_{i=1}^{N}\frac{|s_i|}{2^i}\; n_N(d[s]_{N}) \leq C_0<+\infty,\quad \forall N\in\mathbb{N}^+.
\end{equation}
 Then, we can find a subsequence $\{n_{N_j}\}_{j\geq 1}$ and an infinite dimensional measure $n_{\infty}\in\mathcal{P}(\Rp_\infty)$, such that the following weak limit holds, for every $K$-marginal
\begin{equation}\label{tight-limit}
 n_{N_j}^{(K)}\rightarrow n_{\infty}^{(K)},\quad \text{weakly in $\mathcal{P}(\Rp_K)$},\quad \text{as $N_j$ goes to infinity.}
 \end{equation}
\end{lemma}
\begin{proof}
The uniform bound \eqref{uniform-tight-bound} implies that for each $K$-marginal
\begin{equation} \int_{\Rp_K}\big(\sum_{i=1}^{K}\frac{1}{2^i}|s_i|\big)n_N^{(K)}(d[s]_K)\leq \int_{\Rp_N}\big(\sum_{i=1}^{N}\frac{1}{2^i}|s_i|\big)n_N(d[s]_{N})\leq C_0<+\infty,
\end{equation} 
is uniformly bounded for $N\geq K$. Therefore $\{n_N^{(K)}\}$ is a tight sequence in $\mathcal{P}(\Rp_K)$, which allows us to extract a subsequence $\{n_{N_j}\}_{j\geq 1}$ such that, for a probability measure $n_{\infty}^{(K)}$ on $\Rp_K$,
\begin{equation}\label{tight-limit-pf-tmp1}
 n_{N_j}^{(K)}\rightarrow n_{\infty}^{(K)},\quad \text{weakly in $\mathcal{P}(\Rp_K)$},\quad \text{as $N_j$ goes to infinity}.
\end{equation} 
Starting with $K=1$, by a diagonal argument, we can extract a new subsequence, still denoted as $\{n_{N_j}\}$, such that \eqref{tight-limit-pf-tmp1} holds for every $K\geq 1$. 

It remains to define $n_{\infty}$. To do so, 
we  claim that by our diagonal construction, the following consistency relation holds
\begin{equation}\label{consist-tight-pf-tmp}
\big(n_{\infty}^{(K+1)}\big)^{(K)}=n_{\infty}^{(K)},
\end{equation}
Indeed, by \eqref{tight-limit-pf-tmp1} we have for all $\psi([s]_{K+1})\in C_b^0(\Rp_{K+1})$
\begin{equation*}
    \int_{\Rp_{K+1}}\psi([s]_{K+1})n_{N_j}^{(K+1)}(d[s]_{K+1})\rightarrow  \int_{\Rp_{K+1}}\psi([s]_{K+1})n_{\infty}^{(K+1)}(d[s]_{K+1}),\quad \text{as $N_j$ goes to infinity}.
\end{equation*} Taking $\psi$ that depending only on the first $K$ variables $[s]_K$ in the above equation, and using \eqref{tight-limit-pf-tmp1} for $n_{N_j}^{(K)}$ we obtain \eqref{consist-tight-pf-tmp}.
Thanks to \eqref{consist-tight-pf-tmp} we apply the Kolmogorov extension theorem to obtain an infinite dimensional measure $n_{\infty}\in\mathcal{P}(\Rp_{\infty})$, such that $n_{\infty}^{(K)}$ is indeed the $K$-marginal of $n_{\infty}$.
\end{proof}

Lemma \ref{lm:tightness-R^N} gives a tightness criteria for a sequence of finite-dimensional measures $\{n_N\}$ where each $n_N$ is in $\mathcal{P}(\Rp_N)$. It naturally induces a tightness criteria for a sequence of measures in $\mathcal{P}(\Rp_{\infty})$, where the weak limit is in the sense of Definition \ref{def:weak_continuity_infinite}.

\begin{corollary}[Tightness on $\Rp_\infty$]
 For a sequence of infinite dimensional measures $\{n_{\infty,j}\}_j$ of $\mathcal{P}(\Rp_\infty)$, suppose the uniform bound holds
\begin{equation*}
\int_{\Rp_{\infty}} \sum_{i=1}^{+\infty}\frac{|s_i|}{2^i}\; n_{\infty,j}(d[s]_{\infty})\leq C_0<+\infty, \quad \forall j.
\end{equation*}
Then we can extract a subsequence that weakly converges to some $n_{\infty}\in\mathcal{P}(\Rp_{\infty})$ in the sense of Definition \ref{def:weak_continuity_infinite}.
\end{corollary}

Now we can prove Theorem~\ref{th:weak_limit_of_finite}. 
\begin{proof} [Proof of Theorem \ref{th:weak_limit_of_finite}.]
Thanks to the uniform moment estimate in Lemma \ref{lemma:tightness} and the initial bound \eqref{initial-tight}, we obtain that the uniform-in-$N$ bound \eqref{uniform-tight-bound} holds for $n_N(t)$ at each time $t\geq 0$. Hence, for a fixed time $t$ we can apply Lemma~\ref{lm:tightness-R^N} to obtain a subsequence for which \eqref{eq:weak_convergence_lem_2} holds at time $t$. Then, by a diagonal argument, we can further extract a subsequence such that \eqref{eq:weak_convergence_lem_2} holds for all $t\in [0,+\infty)\cap \mathds{Q}$, a countable many points. At this stage, we could not identify the limit measure as the solution to the infinite problem yet. 

To extend this convergence to arbitrary $t\in [0,+\infty)$, it suffices to show that if \eqref{eq:weak_convergence_lem_2} holds for $\{t_n\}$ with $t_n\rightarrow t$ as $n\rightarrow\infty$, then \eqref{eq:weak_convergence_lem_2} also holds for $t$. This will follow from an equi-continuity estimate. Precisely, we shall show for a fixed $K\geq1$ and $\psi\in C_b^0(\Rp_K)$, the map
\begin{equation}\label{equ-map}
    t \longmapsto\quad \int_{\Rp_N}n_N(t,d[s]_N)\psi([s]_K)
\end{equation} is equi-continuous for $N\geq K$. To this end we shall use the weak formulation \eqref{eq:measure-solution-finite}, which however needs $\psi\in C^1_b(\Rp_K)$. Indeed, for $\psi\in C^1_b(\Rp_K)$, \eqref{eq:measure-solution-finite} implies that the map \eqref{equ-map} is Lipschitz continuous with uniform-in-$N$ Lipschitz constants, since the renewal rate has a uniform upper bound \eqref{as:RR2}. For general $\psi\in C_b^0(\Rp_K)$, we can still obtain the equi-continuity via a density argument.

Now we have the limit for each time $t\geq 0$. By the dominated convergence theorem we derive the integrated-in-time limit \eqref{eq:weak_convergence_lem}. These allows us to pass the limit in the weak formulation from \eqref{eq:measure-solution-finite} to \eqref{eq:measure-solution}. Hence, by the uniqueness in Theorem \ref{th:wellposed_hierarchy}, we can identify the limit measure as the solution to the infinite-times problem. It follows that the subsequential convergence can be improved to the convergence for the full sequence.

\end{proof}
\subsection{Steady states of the infinite-time renewal equation}

By tightness, we can also show the existence of a steady state to the infinite-times equation.

\begin{proposition}\label{prop:stationary-infinite}
Assume \eqref{as:RR1} and \eqref{as:RR2}. Then there exists a steady state for the infinite-times renewal equation.
\end{proposition}

From Definition \ref{def:MSI}, a steady state $n_{\infty}^*\in\mathcal{P}(\Rp_{\infty})$ satisfies
\begin{equation}
    \int_{\Rp_\infty} n_{\infty}(t,d[s]_{\infty})\Big(\sum_{i=1}^{K}\p_{s_i}\psi([s]_{K})+p_{\infty}([s]_\infty)\big(\psi(\tau[s]_K)-\psi([s]_K)\big)\Big)=0,
\end{equation} for all $K\geq 1$ and $\psi([s]_K)\in C_b^1(\Rp_K)$.

\begin{proof}
Theorem~\ref{thm:ststN} gives that for each $N\geq 1$ the $N$-times equation has a unique steady state $n_N^*$, and that these steady states satisfy a uniform moment bound \eqref{uniform-moment-steady}. This allows us to apply Lemma~\ref{lm:tightness-R^N} to subtract a weakly convergent subsequence of $n_{N}^*$ in the sense of \eqref{tight-limit}, whose limit gives a steady state to the infinite-times equation.
\end{proof}

The long term convergence of the evolution problem to the steady state is studied in Section~\ref{sec:uniformError} and Section~\ref{sec:MK-BIG}. In both sections the uniqueness of the steady state is obtained, but additional assumptions, such as \eqref{def-doeblin-cond-sec4} and \eqref{doeblin-constant-bound}, on the renewal rates are needed.

%
\section{Exponential convergence to steady state in Monge-Kantorovich distance}
\label{sec:MK-BIG}

For the $N$-times equation, we have shown the exponential convergence to the steady state via the Doeblin method, in $L^1$ or the total variation distance (Theorem \ref{thm:ststN}). However, the total variation distance does not work for the infinite-times equation, as illustrated by the following example.

Consider two Dirac masses $\delta_{[s^1]_{\infty}}$ and $\delta_{[s^2]_{\infty}}$  concentrated respectively at
\begin{equation*}
[s^1]_{\infty}=(k_1,2k_1,...,Nk_1,...),\quad [s^2]_{\infty}=(k_2,2k_2,...,Nk_2,...),
\end{equation*} with two positive real numbers $k_1\not=k_2$. Let $n_{\infty}^1(t)$ and $n_{\infty}^2(t)$ be the solution to the infinite-times equation \eqref{eq:infinite} with initial data $\delta_{[s^1]_{\infty}}$ and $\delta_{[s^2]_{\infty}}$, respectively. We claim that the total variation distance between $n_{\infty}^1(t)$ and $n_\infty^2(t)$ will always remain $2$. Indeed, as a finite number of jumps do not change the tail behavior, the support of $n_{\infty}^1(t)$ and $n_\infty^2(t)$ will lie in disjoint subsets of $\Rp_{\infty}$. More precisely
\begin{align*}
    &\text{supp }n_{\infty}^1(t)\subseteq \{[s]_{\infty}\in \Rp_{\infty}: \exists N\,\text{ s.t. }\, s_{n+1}-s_n=k_1,\quad \forall n\geq N\},\\
    &\text{supp }n_{\infty}^2(t)\subseteq \{[s]_{\infty}\in \Rp_{\infty}: \exists N\,\text{ s.t. }\, s_{n+1}-s_n=k_2,\quad \forall n\geq N\}.
\end{align*} As $k_1\neq k_2$, the support of $n_{\infty}^1(t)$ of $n_{\infty}^2(t)$ are disjoint therefore the total variation distance
\begin{equation*}
    \|n_{\infty}^1(t)-n_{\infty}^2(t)\|_{\mathcal{M}^1(\Rp_\infty)}=    \|n_{\infty}^1(t)\|_{\mathcal{M}^1(\Rp_\infty)}+\|n_{\infty}^2(t)\|_{\mathcal{M}^1(\Rp_\infty)}=2.
\end{equation*}

In consistency with this example, we also see that the convergence rate in $L^1$ for the $N$-times equation degenerates as $N\rightarrow\infty$ (Section \ref{subsec:strong_condition}). In Section \ref{sec:uniformError} this is compensated by a fast-decay assumption on the renewal rate, to obtain long time convergence for each $N$-marginal. However the result is qualitative and there is no convergence rate.

As discussed above, to obtain exponential convergence, we need a proper metric other than the total variation. This leads us to use  the Monge-Kantorovich distance (M.-K. in short). The M.-K. distance can be applied to PDEs, \cite{FP2020} and, in particular, to structured equations such as the renewal equation, see~\cite{fournier2021non}. Often, a special design of the transport cost is needed to fit the structure of a particular problem.

For the $N$-times equation, we find a suitable cost function, such that the exponential convergence of the M.-K. distance can be obtained, with a \textit{uniform-in-$N$} rate (i.e. does not degenerate as $N\rightarrow\infty$.). This naturally extends to the exponential convergence to the  steady state for the infinite-times equation. The $N$-times case and the infinite-times case are discussed in Section \ref{sec:MK} and Section \ref{sec:infinite_MK_convergence}, respectively, together with certain assumptions needed for the renewal rate.

\subsection{Uniform convergence to steady state in Monge-Kantorovich distance}
\label{sec:MK}

Given a distance $V([s]_N,[s']_N)$ on $\Rp_{N}$ (possibly $N=\infty$, see section~\ref{sec:infinite_MK_convergence}),  also called a cost function, we can define the corresponding M.-K. transport  distance between two probability measures $n_N$ and $m_N$ in $\mathcal{P}(\Rp_{N})$ as
\begin{equation}\label{def-MK-coupling}
\left\{
\begin{split}
& \mathcal{T}_{V}(n_N,m_N):=\inf_{\omega_{N}\in\mathcal{H}(n_N,m_N)}\iint V([s]_N,[s']_N)\omega_{N}(d[s]_N,d[s']_N),
\\[5pt]
& \mathcal{H}(n_N,m_N)=\{\omega_{N}\in\mathcal{P}\big(\Rp_N\times\Rp_N\big)\textup{ with marginals $n_N$ and $m_N$}\}.
\end{split}
\right.
\end{equation}
For more on optimal transport and the M.-K. distance, see e.g. \cite{TopicsVillani2003,Santambrogio_book}.

We are now going to prove the uniform-in-dimension exponential convergence, in the sense of the M.-K. distance. Specifically, the convergence rate does not rely on the dimension of the system, while being related to a variant of the Lipschitz constant for $p_N([s]_N)$.

A key ingredient of the result is to consider the following cost function
\begin{equation} \label{def:TransportFunction}
V_{N,\beta,a}([s]_N,[s']_N):=\sum_{i=1}^{N}\frac{|s_i-s_i'|\wedge a}{(1+\beta)^i} \leq \frac a \beta,
\end{equation} where  $a,\, \beta>0$ are two parameters. Intuitively, it gives less importance to earlier spike times by a factor $\frac{1}{1+\beta}$, and a truncation is imposed to make the cost function bounded. We shall see later in the proof how these designs make a uniform-in-$N$ convergence rate possible.

\begin{theorem}[Exponential contraction in M-K distance] \label{th:finite_monge_convergence}
For the $N$-times renewal equation, we assume that the renewal rate satisfies \eqref{as:RR2} and that we can choose $\beta,a>0$ and $\delta>0$ such that
\begin{equation}\label{eq:condition_of_MK_convergence-1}
\Big|p_N([s]_N)-p_N([s']_N)\Big|\leq \delta \,  V_{N,\beta,a}([s]_N,[s']_N),\qquad \text{for all $[s]_N,[s']_N\in \mathcal{C}_{N}$,}\quad
\end{equation} 
and $\delta>0$ is small enough in the sense
\begin{equation}\label{eq:condition_of_MK_convergence-2}
      \gamma:=\frac{\beta a_-}{1+\beta}-\frac{a\delta}{\beta}>0.
\end{equation}
Then, for two solutions to the $N$-times equation $n_N(t)$ and $m_N(t)$, the following exponential contraction holds
\begin{equation}\label{expoential-contraction}
\mathcal{T}_{V_{N,\beta,a}}(n_N(t),m_N(t))\leq \mathcal{T}_{V_{N,\beta,a}}(n_N(0),m_N(0))e^{-\gamma t},
\end{equation}
where $\mathcal{T}_{V_{N,\beta,a}}$ is the M.-K. distance induced by the cost function \eqref{def:TransportFunction}. In particular, \eqref{expoential-contraction} implies the exponential convergence to the steady state with rate $\gamma>0$.
\end{theorem}

The constant $\delta$ in \eqref{eq:condition_of_MK_convergence-1} can be understood as the Lipschitz constant of $p_N([s]_N)$ relative to $V_{N,\beta,a}([s]_N,[s']_N)$. It is supposed to be small enough to ensure the convergence rate $\gamma>0$ in \eqref{eq:condition_of_MK_convergence-2}. In particular, if we can find fixed $\beta,a>0$ and $\delta>0$ such that \eqref{eq:condition_of_MK_convergence-1}-\eqref{eq:condition_of_MK_convergence-2} hold for all $N\geq1$, then long term exponential convergence holds for the $N$-times equation, with a uniform-in-$N$ convergence rate $\gamma>0$.

 We first present the proof of Theorem \ref{th:finite_monge_convergence}, after which we will discuss explicit conditions on the renewal rate to ensure \eqref{eq:condition_of_MK_convergence-1}-\eqref{eq:condition_of_MK_convergence-2} and give examples.

\begin{proof}[Proof of Theorem \ref{th:finite_monge_convergence}]

\textit{Step 1. The coupling evolution.} 
We use the coupling method. Given two initial data $n_N(0)$ and $m_N(0)$, and one initial coupling $\omega_N(0)\in\mathcal{H}(n_N(0),m_N(0))$ (see \eqref{def-MK-coupling} for the definition), we want to determine an evolution of the coupling measure $w_N(t)$. In other words, for each time $t$, $\omega_N(t)$ is a coupling between $n_N(t)$ and $m_N(t)$. By definition of the M.-K. distance \eqref{def-MK-coupling} we have,
\begin{equation}\label{MK-Coupling-t}
\mathcal{T}_{V_{N,\beta,a}}\big(n_N(t),m_N(t)\big)\leq \int_{\mathcal{C}_N\times\mathcal{C}_N}V_{N,\beta,a}([s]_N,[s']_N)\omega_N(t,d[s]_N,d[s']_N)=:T_{V_{N,\beta,a}}(\omega_N),
\end{equation}where we define $T_{V_{N,\beta,a}}(\omega_N)$ as the transport cost of coupling measure $\omega_N$. The evolution of $\omega_N(t)$ shall be constructed in such a way that we can estimate the right hand side of \eqref{MK-Coupling-t}.

To find such a proper coupling $\omega_N(t)$, we choose a strategy following \cite{fournier2021non}. We define $\omega_N(t)$ as the weak solution of the following equation
\begin{align}\label{eq:evolution_of_coupled_measure}
 \p_t\omega_N(t,[s]_N,[s']_N)+&\sum_{i=1}^{N}\Big((\p_{s_i}+\p_{s_i'})\omega_N(t,[s]_N,[s']_N)\Big)+\max\big\{p_{N}([s]_N),p_N([s']_N)\big\}\omega_N(t,[s]_N,[s']_N) \notag\\
=&\delta_0(s_1)\delta_0(s_1') \iint \min\big\{p_N([s]_{2,N},u),p_N([s']_{2,N},u')\big\}\omega_N(t,[s]_{2,N},du,d[s']_{2,N},du') \notag \\
&+\delta_0(s_1)\int \big(p_N([s]_{2,N},u)-p_N([s']_N)\big)_{+}\omega_N(t,[s]_{2,N},du,[s']_N)\\
&+ \delta_0(s_1')\int \big(p_N([s]_N)-p_N([s']_{2,N},u')\big)_{+}\omega_N(t,[s]_N,[s']_{2,N},du').\notag
\end{align}
In terms of weak solution, this means we want find a $\omega_N(t)$ such that, for arbitrary $0\leq T<+\infty$ and test function $\psi\in C_b^1\big([0,T]\times\mathcal{C}_N\times\Rp_N\big)$,
\begin{equation}\label{eq:coupling_evolution_weak}
\begin{split}
\int_0^T& \! \!\iint_{\mathcal{C}_N\times\mathcal{C}_N} \psi(t,[s]_N,[s']_N)\omega_N(t,d[s]_N,d[s']_N)dt\\
=&\int_0^T \! \! \iint_{\mathcal{C}_N\times\mathcal{C}_N}\sum_{i=1}^{N}\Big(\frac{\p}{\p_{s_i}}+\frac{\p}{\p_{s_i'}}\Big) \psi(t,[s]_N,[s']_N)\omega_N(t,d[s]_N,d[s']_N)dt\\
&+\int_0^T \! \! \iint_{\mathcal{C}_N\times\mathcal{C}_N}\Big(\Big[\psi(t,\tau[s]_N,\tau[s']_N)-\psi(t,[s]_N,[s']_N)\Big]\min\big\{p_N([s]_N), p_N([s']_N)\big\}\\
&+ \Big[\psi(t,\tau[s]_N,[s']_N)-\psi(t,[s]_N,[s']_N)\Big]\big(p_N([s]_N)- p_N([s']_N)\big)_{+}\\
&+\Big[\psi(t,[s]_N,\tau[s']_N)-\psi(t,[s]_N,[s']_N)\Big]\big(p_N([s']_N)- p_N([s]_N)\big)_{+}\Big)\omega_N(t,d[s]_N,d[s']_N)dt.
\end{split}
\end{equation}

\textit{Microscopic description.} The above equation can be seen as the evolution of the probability distribution for a stochastic process on $\mathcal{C}_N\times \mathcal{C}_N$. Given a particle at $([s]_N,[s']_N)$, it will jump to $(\tau[s]_N,\tau[s']_N)$ with rate $\min\{p_N([s]_{N}),p_N([s']_{N})\}$, jump to $(\tau[s]_N,[s']_N)$ with rate $\big(p_N([s]_{N})-p_N([s']_N)\big)_{+}$ and jump to $([s]_N,\tau[s']_N)$ with rate $\big(p_N([s']_N)-p_N([s]_N)\big)_+$ (we recall the shift operator $\tau$ is defined by Equation~\eqref{def:shift_operator}). Considering the marginal distribution, we can see that given a particle at $[s]_N$, it will jump to $\tau[s]_N$ with rate $p_N([s]_N)$, because,
\begin{equation}\label{eq:marginal_summation}
\min\{p_N([s]_{N}),p_N([s']_{N})\}+\big(p_N([s]_{N})-p_N([s']_N)\big)_{+}=p_N([s]_N).
\end{equation}

\textit{Existence of evolution.} The existence of a solution to Equation~\eqref{eq:coupling_evolution_weak} is similar to that in \cite[Section 4]{fournier2021non}, which uses a prior tightness estimate to treat the coupling of the two-times equation.

\textit{Proof strategy for Theorem~\ref{th:finite_monge_convergence}.} Working under the microscopic description above, we can see that if a coupled particle at $([s]_N,[s']_N)$ jumps to $(\tau[s]_N,\tau[s']_N)$, the transport cost will decay as follows,
\begin{equation}
V_{N,\beta,a}(\tau[s]_N,\tau[s']_N)\leq \frac{1}{1+\beta} V_{N,\beta,a}([s]_N,[s']_N).
\end{equation}
This is essentially the reason of the exponential convergence in the M.-K. distance. Meanwhile, we also need to control the asynchronous jump from $([s]_N,[s']_N)$ to $(\tau[s]_N,[s']_N)$ or $([s]_N,\tau[s']_N)$, with rate $\big|p_N([s]_N)-p_N([s']_N)\big|$. This is done by controlling $\big|p_N([s]_N)-p_N([s']_N)\big|$ using $V_{N,\beta,a}([s]_N,[s']_N)$, thanks to Assumption \eqref{eq:condition_of_MK_convergence-1}.

\textit{Step 2. Estimate of M.-K. distance.} 
We now estimate the M.-K. distance in the spirit of controlling asynchronous jump mentioned above.

Although $V_{N,\beta,a}$ is not $C^1$, we can still use it as a test function in~\eqref{eq:coupling_evolution_weak}, by regularization arguments. Also, the directional derivative term  $\Big(\frac{\p}{\p_{s_i}}+\frac{\p}{\p_{s_i'}}\Big) V_{N,\beta,a}([s]_N,[s']_N)$ vanishes because  $V_{N,\beta,a}([s]_N,[s']_N)$ depends only on each $s_i-s'_i$, and we find
\begin{equation}\label{eq:coupling_evolution}
\begin{split}
\frac{d}{dt}&\iint_{\mathcal{C}_N\times\mathcal{C}_N} V_{N,\beta,a}([s]_N,[s']_N)\omega_N(t,d[s]_N,d[s']_N)\\
=&\iint_{\mathcal{C}_N\times\mathcal{C}_N}\Big(\Big[V_{N,\beta,a}(\tau [s]_N,\tau [s']_N)-V_{N,\beta,a}([s]_N,[s']_N)\Big]\big(p_N([s]_N)\wedge p_N([s']_N)\big)\\
&+\Big[V_{N,\beta,a}(\tau [s]_N,[s']_N)-V_{N,\beta,a}([s]_N,[s']_N)\Big]\big(p_N([s]_N)- p_N([s']_N)\big)_{+}\\
&+\Big[V_{N,\beta,a}([s]_N,\tau [s']_N)-V_{N,\beta,a}([s]_N,[s']_N)\Big]\big(p_N([s']_N)- p_N([s]_N)\big)_{+}\Big)\omega_N(t,d[s]_N,d[s']_N).
\end{split}
\end{equation}

For the second last line in Equation~\eqref{eq:coupling_evolution}, using $0\leq V_{N,\beta,a} \leq \frac{a}{\beta}$ and \eqref{eq:condition_of_MK_convergence-1}, we write
\begin{equation*}
\begin{split}
& \Big|V_{N,\beta,a}(\tau [s]_N,[s']_N)-V_{N,\beta,a}([s]_N,[s']_N)\Big|\leq \frac{a}{\beta},\\
& \Big|V_{N,\beta,a}([s]_N,\tau[s']_N)-V_{N,\beta,a}([s]_N,[s']_N)\Big|\leq \frac{a}{\beta},\\
& (p_N([s]_N)-p_N([s']_N))_++(p_N([s']_N)-p_N([s]_N))_+\leq \delta V_{N,\beta,a}([s]_N,[s']_N).
\end{split}
\end{equation*}
For the third last line in Equation~\eqref{eq:coupling_evolution}, we notice that
\begin{equation*}
\begin{split}
&V_{N,\beta,a}(\tau [s]_N,\tau [s']_N)-V_{N,\beta,a}([s]_N,[s']_N)=\frac{\beta}{1+\beta}V_{N,\beta,a}([s]_N,[s']_N),\\
& p_N([s]_N)\wedge p_N([s']_N)\geq a_-.
\end{split}
\end{equation*}
Combining equations above, we can write Equation~\eqref{eq:coupling_evolution} as
\begin{equation*}
\begin{split}
\frac{d}{dt}\int_{\Rp_N\times\Rp_N} &V_{N,\beta,a}([s]_N,[s']_N)\omega_N(t,d[s]_N,d[s']_N)\\
\leq & -\iint_{\Rp_N\times\Rp_N}\frac{\beta a_-}{1+\beta}V_{N,\beta,a}([s]_N,[s']_N)\omega_N(t,d[s]_N,d[s']_N)\\
&+\iint_{\Rp_N\times\Rp_N}\frac{a\delta}{\beta}V_{N,\beta,a}([s]_N,[s']_N)\omega_N(t,d[s]_N,d[s']_N).
\end{split}
\end{equation*}
Recalling the definition of $T_{V_{N,\beta,a}}(\omega_N)$ in Equation \eqref{MK-Coupling-t}, this leads to,
\begin{equation*}
\frac{d}{dt}T_{V_{N,\beta,a}}\big(\omega_N(t)\big)\leq T_{V_{N,\beta,a}}\big(\omega_N(t)\big)\Big(-\frac{\beta a_-}{1+\beta}+\frac{a\delta}{\beta}\Big)= -\gamma  T_{V_{N,\beta,a}}\big(\omega_N(t)\big),
\end{equation*}
with $\gamma >0$ defined in \eqref{eq:condition_of_MK_convergence-2}. 
By Gronwall's inequality and \eqref{def-MK-coupling} we have
\begin{align*}
\mathcal{T}_{V_{N,\beta,a}}\big(n_N(t),m_N(t)\big)&\leq T_{V_{N,\beta,a}}\big(\omega_N(t)\big)\\&\leq e^{-\gamma t}T_{V_{N,\beta,a}}\big(\omega_N(0)\big)\\&=e^{-\gamma t}\iint V_{N,\beta,a}([s]_N,[s']_N)\omega_N(0,d[s]_N,d[s']_N).
\end{align*}
Finally, taking the infimum among all initial coupling $w_N(0)\in\mathcal{H}(n(0),m(0))$ in the right of the above equation, we conclude the proof by definition of the M.-K. distance \eqref{def-MK-coupling}.
\end{proof}
%
To give an explicit and sufficient condition to satisfy Assumptions \eqref{eq:condition_of_MK_convergence-1}-\eqref{eq:condition_of_MK_convergence-2}, we first define some constants. For $\beta>0$, we define the weighted maximum Lipschitz constant of $\varphi_i([s]_i)$ $L_N$ and and the weighted maximum of fluctuation $F_N$ as
\begin{equation}\label{eq:modified_lipschitz}
\begin{cases}
L_N(\beta):=\max_{1\leq i\leq N}\sup_{[s]_i,[s']_i}(1+\beta)^i\frac{|\varphi_i([s]_i)-\varphi_i([s']_i)|}{|s_i-s_i'|},
\\[5pt]
F_N(\beta):=\max_{1\leq i\leq N}\sup_{[s]_i,[s']_i}(1+\beta)^i\Big|\varphi_i([s]_i)-\varphi_i([s']_i)\Big|.
\end{cases}
\end{equation}
\begin{lemma}\label{lem:decayingg_Lipschitz}
Under Assumptions \eqref{as:RR1} and \eqref{as:RR2} and for arbitrary $1\leq N\leq +\infty$, if we have a proper $\beta>0$ and a small enough $a>0$ such that,
\begin{equation}\label{eq:decaying_lipschitz}
F_N(\beta)<\frac{a_- \beta^2}{1+\beta},\qquad \textup{and} \qquad a L_N(\beta)<\frac{a_- \beta^2}{1+\beta},
\end{equation}
then Assumptions \eqref{eq:condition_of_MK_convergence-1}-\eqref{eq:condition_of_MK_convergence-2} hold with $\delta$ and $\gamma$ defined as,
\begin{equation}\label{eq:delta}
\delta:=\max\{\frac{F_N(\beta)}{a},L_N(\beta)\}, \qquad \gamma:=\frac{\beta a_-}{1+\beta}-\frac{F_N(\beta)\vee a L_N(\beta)}{\beta}.
\end{equation}
\end{lemma}
\begin{proof}
Indeed, by Assumption \eqref{as:RR2} we have
\begin{equation*}
\begin{split}
\frac{|p_N([s]_N)-p_N([s']_N)|}{V_{N,\beta,a}([s]_N,[s']_N)}\leq \frac{\sum_{i=1}^{N}|\varphi_i([s]_i)-\varphi_i([s']_i)|}{\sum_{i=1}^{N}\frac{1}{(1+\beta)^i}(|s_i-s_i'|\wedge a)}&\leq  \max_{1\leq i\leq N}\Big\{\frac{|\varphi_i([s]_i)-\varphi_i([s']_i)|}{\frac{1}{(1+\beta)^i}(|s_i-s_i'|\wedge a)}\Big\}\\
&\leq \frac{F_N(\beta)}{a}\vee L_N(\beta):= \delta.
\end{split}
\end{equation*}
Then the expression \eqref{eq:delta} of $\delta$, and thus \eqref{eq:condition_of_MK_convergence-1}-\eqref{eq:condition_of_MK_convergence-2}, follow.
\end{proof}
While the Assumption \eqref{eq:decaying_lipschitz} may seem restrictive, we give an explicit example here for $N=+\infty$, and the examples of $N<+\infty$ can be readily derived from it. It also serves as an example for the Assumption \eqref{eq:MK_conditions_infinite} in Theorem \ref{th:infinite_MK_convergence}. This explicit example is 
\begin{equation*}
p_{\infty}([s]_{\infty})=\sum_{i=1}^{+\infty}\frac{(a_-\vee s_i)\wedge C a_-}{(1+\beta)^i}.
\end{equation*}
Here $C>1$ must satisfy a proper condition that we explain later. In this example, there is a subtle trade-off between the fluctuation of $\varphi_i([s]_i)$ relative to $[s]_i$, and the exponential decay of this fluctuation relative to $i$. With the notation~\eqref{eq:modified_lipschitz} extended to $N=\infty$,  we have $F_\infty(\beta)=(C-1)a_-$ and $L_\infty(\beta)\leq1$. When  $C-1$ is large, the assumption $C-1\leq \frac{\beta^2}{1+\beta}$ requires that the parameter  $\beta$ is also large. More generally, given any positive Lipschitz function $f(s)$ defined for $s\geq 0$, we can use the renewal rate as the following where $\eps$ is a small enough positive number,
\begin{equation*}
p_{\infty}([s]_{\infty})=\sum_{i=1}^{+\infty}\frac{1}{(1+\beta)^i}\Big[\big(a_-\vee f(s_i)\big)\wedge \big(\frac{\beta^2}{1+\beta}+1-\eps\big)a_-\Big].
\end{equation*}

\subsection{Exponential convergence for measure solutions in $\Rp_\infty$}
\label{sec:infinite_MK_convergence}
We now extend the exponential convergence result in Section~\ref{sec:MK} to the infinite-times renewal equation. 
\begin{theorem}\label{th:infinite_MK_convergence}\textup{\textbf{(Exponential Convergence of Measure Solutions)}}
Assume the renewal rate satisfies \eqref{as:RR1} and \eqref{as:RR2}. We also assume that we can choose $\beta,a>0$ and $\delta>0$ such that for arbitrary $[s]_{\infty},[s']_{\infty}\in\mathcal{C}_{\infty}$,
\begin{equation}\label{eq:MK_conditions_infinite}
\Big|p_\infty([s]_\infty)-p_\infty([s']_\infty)\Big|\leq \delta \,  V_{\infty,\beta,a}([s]_\infty,[s']_\infty)\qquad \textup{and}\qquad \gamma:=\frac{\beta a_-}{1+\beta}-\frac{a\delta}{\beta}>0.
\end{equation}
Then for the infinite-times equation, given two different initial probability distributions $n_{\infty}(0),m_{\infty}(0)\in \mathcal{P}(\Rp_{\infty})$, we have the corresponding global-in-time solutions  $n_{\infty}(t)$ and $m_{\infty}(t)$ as in Definition \ref{def:MSI}. The two solutions satisfy the following exponential convergence of M.-K. distance,
\begin{equation*}
\mathcal{T}_{V_{\infty,\beta,a}}\big(n_\infty(t),m_\infty(t)\big)\leq \mathcal{T}_{V_{\infty,\beta,a}}\big(n_\infty(0),m_\infty(0)\big)e^{-\gamma t}.
\end{equation*}
\end{theorem}

Notice that by replacing the $N<+\infty$ in the Assumptions \eqref{eq:condition_of_MK_convergence-1}-\eqref{eq:condition_of_MK_convergence-2} of Theorem~\ref{th:finite_monge_convergence} with $N=+\infty$, we recover  Assumption \eqref{eq:MK_conditions_infinite} of Theorem \ref{th:infinite_MK_convergence}. 

In principle, Theorem~\ref{th:infinite_MK_convergence} can be proved via taking the limit $N\rightarrow\infty$ in Theorem~\ref{th:finite_monge_convergence}, using the local-in-time strong convergence in Theorem \ref{th:wellposed_hierarchy}. However, here we rather want to give a direct proof in the same way to the finite-times case. To this end,  we treat the solution as an infinite dimension measure as in Definition \ref{def:MSI}, which is equivalent to the hierarchy viewpoint thanks to Lemma \ref{lemma:equivalence}.

\begin{proof}[Proof of Theorem~\ref{th:infinite_MK_convergence}]

We extend the coupling method in Theorem~\ref{th:infinite_MK_convergence} to the infinite-times case. First we want to build an evolution of the coupling measure $\omega_\infty(t)\in C_w\big([0,+\infty);\mathcal{M}(\Rp_\infty\times\Rp_\infty)\big)$, which generalizes \eqref{eq:coupling_evolution_weak} to the infinite-times case. The existence of $\omega_\infty(t)$ can be obtained by letting $N\rightarrow\infty$ for the $N$-times coupling measure $\omega_N(t)$ given in the proof of Theorem \ref{th:finite_monge_convergence}, generalizing the local-in-time strong limits in Theorem \ref{th:uniqueness} and \ref{th:wellposed_hierarchy} to the coupling measures $\omega_N(t)$. 


Based on the arguments above, we have an element $\omega_{\infty}\in C_w\big([0,+\infty);\mathcal{P}(\Rp_\infty\times\Rp_\infty)\big)$ satisfying,
\begin{equation*}
\begin{split}
\frac{d}{dt}&\iint_{\Rp_\infty\times\Rp_\infty} V_{\infty,\beta,a}([s]_{\infty},[s']_{\infty})\omega_{\infty}(t,d[s]_{\infty},d[s']_{\infty})\\
=&\iint_{\Rp_\infty\times\Rp_\infty}\sum_{i=1}^{+\infty}\Big(\frac{\p}{\p_{s_i}}+\frac{\p}{\p_{s_i'}}\Big) V_{\infty,\beta,a}([s]_\infty,[s']_\infty)\omega_\infty(t,d[s]_\infty,d[s']_\infty)\\
+&\iint_{\Rp_{\infty}\times\Rp_{\infty}}\Big(\Big[V_{\infty,\beta,a}(\tau [s]_\infty,\tau [s']_\infty)-V_{N,\beta,a}([s]_\infty,[s']_\infty)\Big]\min\big\{p_\infty([s]_\infty),p_\infty([s']_\infty)\big\}\\
+&\Big[V_{\infty,\beta,a}(\tau [s]_\infty,[s']_\infty)-V_{\infty,\beta,a}([s]_\infty,[s']_\infty)\Big]\big(p_\infty([s]_\infty)- p_\infty([s']_\infty)\big)_{+}\\
+&\Big[V_{\infty,\beta,a}([s]_\infty,\tau [s']_\infty)-V_{\infty',\beta,a}([s]_\infty,[s']_\infty)\Big]\big(p_\infty([s']_\infty)- p_\infty([s]_\infty)\big)_{+}\Big)\omega_\infty(t,d[s]_\infty,d[s']_\infty).
\end{split}
\end{equation*}
Using the same method as  in Theorem~\ref{th:finite_monge_convergence} we find,
\begin{equation*}
\mathcal{T}_{V_{\infty,\beta,a}}\big(n_\infty(t),m_\infty(t)\big)\leq \mathcal{T}_{V_{\infty,\beta,a}}\big(n_\infty(0),m_\infty(0)\big)e^{-\gamma t}.
\end{equation*}
This concludes the proof.
\end{proof}
\section{Conclusion and discussions}

In Section \ref{sec:MK-BIG} we revisit the problem of long time behavior via a rather different approach. In particular, by using a suitable Monge-Kantorovich distance (instead of the strong $L^1$-distance), we can give a uniform-in-$N$ convergence rate. However, we still need more assumptions on the renewal rate than \eqref{as:RR1} and \eqref{as:RR2}. On the one hand, the question of long time behavior in more general cases, seems a difficult task; One the other hand, there has been many results (for example~\cite{GRAHAM_Hawkes_2019,Costa_Hawkes_2020}) about the long time behaviour of the Hawkes process, which is formally similar to a particular kind of infinite-times equations. We can ask if there are other types of assumptions, different from \eqref{as:RR1} and \eqref{as:RR2}, where there are other formulations facilitating the proof of convergence for infinite-times equations.

Concerning the possible nonlinearity in renewal equations, we can ask what kinds of nonlinearity would be suitable and meaningful for this topic. After a suitable choice of nonlinearity, questions about well-posedness and long-time behaviour could also be raised.

The formulation of hierarchy system \eqref{eqInfinity} is similar to the BBGKY hierarchy, while rather than infinite particles, we couple an infinite number of state variables here. However we can still ask if it's possible to have some kinds of 'propagation of chaos' in the infinite-times equation. Since in the present formulation it's necessary that $s_1\leq s_2\leq...$, the possibility of 'propagation of chaos' may be searched for based on some additional reformulations. Meanwhile in a heuristic way, we could regard the state variables $s_1\leq s_2\leq ...$ as an infinite number of particles with a strict order structure. This could be a starting point of exploring how to use infinite-dimensional PDE to describe a particle system with heterogeneous particles.

\appendix

\section{Doeblin Theorem}\label{sec:Doeblin_Theorem}

\begin{definition}
\textbf{(Markov semi-group).} Let $(\mathcal{X},\mathcal{A})$ be a measure space and $P_t:\mathcal{P}(\mathcal{X})\rightarrow \mathcal{P}(\mathcal{X})$ be a linear semi-group. We say that $P_t$ is a Markov semi-group if $P_t\mu\geq 0$ for all $\mu\geq 0$ and $\int_{X}P_t\mu=\int_{X}\mu$ for all $\mathcal{P}(\mathcal{X})$. In other words, $(P_t)$ preserves the subset of probability measure $\mathcal{P}(\mathcal{X})$.
\end{definition}
\begin{definition}
\textbf{(Doeblin's condition).} Let $P_t:\mathcal{P}(\mathcal{X})\rightarrow\mathcal{P}(\mathcal{X})$ be a Markov semi-group. We say that $(P_t)$ satisfies Doeblin's condition if there exist $t_0>0$ and $\nu\in\mathcal{P}(\mathcal{X})$ such that
\begin{equation} 
P_{t_0}\mu\geq \alpha\nu,\quad \forall \mu\in \mathcal{P}(\mathcal{X})
\end{equation}
\end{definition}
And then we state the Doeblin Theorem.
\begin{theorem}
\textbf{(Extended Doeblin's Theorem).} Let $P_t:\mathcal{P}(\mathcal{X})\rightarrow \mathcal{P}(\mathcal{X})$ be a Markov semi-group that satisfies Doeblin's condition. Then the semigroup has a unique equilibrium $\mu^*\in\mathcal{P}(\mathcal{X})$, Moreover, for all $\mu\in\mathcal{P}(\mathcal{X})$ we have
\begin{equation}
\lVert P_t\mu -\langle \mu \rangle \mu^*\rVert_{\mathcal{M}^1}\leq \frac{1}{1-\alpha}e^{-\lambda t}\lVert \mu-\langle \mu\rangle\mu^*\rVert_{\mathcal{M}^1},\quad \forall t\geq 0
\end{equation}
with $\langle \mu \rangle:=\int_{X}\mu$ and $\lambda=-\frac{\ln(1-\alpha)}{t_0}$.

When the semi-group maps $L^1$ into itself, the same result holds with $L^1$ in place of $\mathcal{P}$. Then, $\mu^* \in L^1$ and $\nu$ is automatically an $L^1$ function.
\end{theorem}
For a proof of Doeblin's theorem, the readers may see \cite{gabriel2018measure}. With a slight modification, we can prove the extended Doeblin's theorem.

\section{The Kolmogorov extension Theorem}
\label{ap:K}

Denote by $\mathcal{R}^N$ the set of Borel sets in $\R^N$ with $N=1,2,...,\infty$, as a $\sigma$-algebra. Specifically, we define $\mathcal{R}^{\N}$ to be the set of Borel sets generated by the product topology on $\mathds{R}^\mathds{N}$, where we use $\N$ instead of $\infty$ as the superscript to emphasize the countable dimension on the space $\R^{\N}$.
\begin{theorem}[Kolmogorov's extension theorem] \label{th:Kolmogorov_Extension}
Assume we are given a sequence of probability measures $\mu_{N}$ on $(\R^N,\mathcal{R}^N)$ that is  consistent, which means $\mu_{N+1}^{(N)}=\mu_{N}$.
Then, there is a unique probability measure $\mu_{\infty}$ on $(\R^\N,\mathcal{R}^\N)$ such that for any $K$,
\begin{equation}
\mu_{\infty}^{(K)}=\mu_{K}.
\end{equation}
\end{theorem}
Its proof can be found in \cite{TaoMeasure, Durrett_book} for instance. We can use this extension theorem to construct a solution of infinite-times renewal equation from a hierarchy solution. This construction is thus a part of the equivalence between the two kinds of solution as shown in Lemma \ref{lemma:equivalence}.

\section*{Acknowledgements}
ZZ has received support from   the National Key R\&D Program of China, Project Number 2021YFA1001200, 2020YFA0712000,  and the NSFC grant, Project Number 12031013, 12171013. BP has received funding from the European Research Council (ERC) under the European Union's Horizon 2020 research and innovation programme grant agreement No 740623. DS has received support from ANR ChaMaNe No: ANR-19-CE40-0024.

\section*{ORCID iDs}
Benoît Perthame: \href{https://orcid.org/0000-0002-7091-1200}{https://orcid.org/0000-0002-7091-1200}.\\
Delphine Salort: \href{https://orcid.org/0000-0003-4289-0286}{https://orcid.org/0000-0003-4289-0286}.\\
Zhennan Zhou: \href{https://orcid.org/0000-0003-4289-0286}{https://orcid.org/0000-0003-4822-0275}.


\end{document}